\documentclass[11pt]{amsart}
\usepackage[a4paper,left=3cm,right=3cm]{geometry}
\usepackage[latin1]{inputenc}
\usepackage[english]{babel}
\usepackage{amsmath,graphicx,graphics}
\usepackage{verbatim}
\usepackage{amssymb}
\usepackage{color}
\usepackage{hyperref}
\usepackage{amsthm}
\usepackage{tikz}
\usepackage{pgfplots}
\usepackage{tikz-3dplot}
\usepackage{float}
\usepackage{amsfonts} 
\usepackage{enumerate}
\usepackage{subcaption}
\usepackage{outlines}

\newtheorem{theorem}{Theorem}[section]

\newtheorem{corollary}{Corollary}[theorem]
\newtheorem{lemma}[theorem]{Lemma}
\newtheorem{proposition}[theorem]{Proposition}
\newtheorem{remark}[theorem]{Remark}
\newtheorem{definition}[theorem]{Definition}

\begin{document}
\title{Optimal bounds for Neumann eigenvalues in terms of the diameter}
\author[]{Antoine Henrot, Marco Michetti}

\address[Antoine Henrot]{Universit\'e de Lorraine, CNRS, IECL, F-54000 Nancy, France}
\email{antoine.henrot@univ-lorraine.fr}
\address[Marco Michetti]{
Universit\'e Paris-Saclay, CNRS, Laboratoire de Math\'ematiques d'Orsay, F-91450 Orsay France}
\email{marco.michetti@universite-paris-saclay.fr}

\date{\today}
\begin{abstract}
    In this paper, we obtain optimal upper bounds for all the Neumann eigenvalues in two situations (that are closely related).
    First we consider a one-dimensional Sturm-Liouville eigenvalue problem where the density is a function $h(x)$ whose
    some power is concave. We prove existence of a maximizer for $\mu_k(h)$ and we completely characterize it.
    Then we consider the Neumann eigenvalues (for the Laplacian) of a domain $\Omega\subset
    \mathbb{R}^d$ of given  diameter and we assume that its profile function (defined as the $d-1$ dimensional measure
    of the slices orthogonal to a diameter) has also some power that is concave.  This includes the case of convex domains
    in $\mathbb{R}^d$, containing and generalizing previous results by P.  Kr\" oger. On the other hand, in the last section,
    we give examples of domains for which the upper bound fails to be true, showing that, in general, 
    $\sup D^2(\Omega)\mu_k(\Omega)= +\infty$.
\end{abstract}

\maketitle

\tableofcontents

\section{Introduction}

Let $\Omega\subset \mathbb{R}^d$ be an open, bounded and connected Lipschitz set
(or more generally a domain for which the embedding $H^1(\Omega) \hookrightarrow L^2(\Omega)$
is compact), the Neumann eigenvalue problem on $\Omega$ consists in solving the eigenvalue problem 
\begin{equation*}
\begin{cases}
     -\Delta u=\mu u\qquad  &\Omega \\
      \partial_{\nu} u=0 \qquad &\partial\Omega
\end{cases}
\end{equation*}
where $ \partial_{\nu} u$ denotes the normal derivative of $u$, this PDE has to be understood in the usual weak sense.
As the Sobolev embedding $H^1(\Omega) \rightarrow L^2(\Omega)$ is compact here, the spectrum of the Neumann problem is discrete and the eigenvalues (counted with their multiplicities) go to infinity
$$0= \mu_0(\Omega)< \mu_1(\Omega)\le  \mu_2 (\Omega) \le \cdots \rightarrow +\infty.$$
We also have a variational characterization of the Neumann eigenvalues, for $k\geq 1$:
\begin{equation}\label{eVFN}
\mu_k(\Omega)=\min_{E_k} \max_{0\neq u\in E_k} \frac{\int_{\Omega}|\nabla u|^2dx}{\int_{ \Omega}u^2 dx},
\end{equation}
where the minimum is taken over all $k-$dimensional subspaces of the Sobolev space $H^1(\Omega)$ which are $L^2-$orthogonal to constants on $\Omega$.

A central question in spectral geometry is to find upper bounds for the Neumann eigenvalues involving geometric quantities related to the domain. The best result that one can achieve in this direction is to prove an explicit and sharp upper bound for the eigenvalues and to give a characterization of the domains for which the equality is achieved.\\
One of the first result in this direction is the Szeg\"o-Weinberger inequality:
\begin{equation*}
|\Omega|^{\frac{2}{d}}\mu_1(\Omega)\leq |B|^{\frac{2}{d}}\mu_1(B),
\end{equation*}
where $B$ is a ball. This inequality was first proved by G. Szeg\"o in \cite{Szeg54} for planar, simply connected and Lipschitz sets and then H. F. Weinberger in \cite{W56} removed the dimensional and topological constraints. An explicit and sharp upper bound involving the volume of the domain is also known for the second Neumann eigenvalue, this was proved in $2$ dimensions for simply connected domains by A. Girouard, N. Nadirashvili and I. Polterovich in \cite{GNP09} and in all dimensions without any restriction on the topology
by D. Bucur and the first author in \cite{BH19}. The inequality reads as follows:
\begin{equation*}
|\Omega|^{\frac{2}{d}}\mu_2(\Omega)\leq |\Omega^*|^{\frac{2}{d}}\mu_2(\Omega^*),
\end{equation*}
where $\Omega^*$ is the union of two disjoint equal balls. The existence of a maximizer for the quantity $|\Omega|^{\frac{2}{d}}\mu_k(\Omega)$ with $k\geq 3$ is not known in general, this problem is studied in \cite{BMO22}.

Another related question is to find inequalities for the Neumann eigenvalues involving the perimeter of the set. Recently the first author, A. Lemenant and I. Lucardesi in \cite{HLL21} proved that for every $\Omega$ plane convex domain with two axis of symmetry the following inequality is true:
\begin{equation*}
P(\Omega)^2\mu_1(\Omega)\leq 16\pi^2,
\end{equation*}   
with equality for a square and an equilateral triangle. It is conjectured that this inequality holds true for any plane convex domains.
In higher dimensions, nothing seems to be known for the analogous quantity $P^{\frac{2}{d-1}}(\Omega) \mu_1(\Omega)$.
Note that, without convexity constraint, we have $\sup P^{\frac{2}{d-1}}(\Omega) \mu_1(\Omega) = +\infty$ in any dimension.

In this work we study the maximization problem for Neumann eigenvalue under diameter constraint.  
We denote by $D(\Omega)$ the diameter of $\Omega$. Results in this direction is given by S. Y. Cheng in \cite{Cheng75}, where he gives general  upper bounds involving the diameter for smooth and complete Riemannian manifolds. The given bound is sharp for $\mu_1$ but not for the other eigenvalues. Later, this sharp and explicit upper bound for the first eigenvalue has been
generalized by R. Banuelos and K. Burdzy in \cite{BB99} to (slightly) more general domains than convex ones. Finally P. Kr\"oger \cite{K99} prove sharp upper bounds for convex domains in all dimensions. In this work we prove sharp upper bounds for a more general class of domains. In order to describe the main result we first need the following definitions
\begin{definition}
let $h$ be a non negative bounded function, then we say that $h$ is $\beta-$concave if $\beta>0$ is the largest number for 
which $h^{\beta}$ is concave (note that if $h^\beta$ is concave for some $\beta >0$, then $h^{\beta_1}$ is concave for any
$0\leq \beta_1 \leq \beta$).
\end{definition}
In the sequel, we will systematically put the first coordinate $x_1$ along one diameter of our domain $\Omega$.
\begin{definition}\label{dprofile}
Let $\Omega\subset \mathbb{R}^d$ be a domain, the profile function $g$ associated to $\Omega$ is the function defined in the following way:
\begin{equation*}
g(x_1)=\mathcal{H}^{d-1}(\{x'\in \mathbb{R}^{d-1}\;|\; (x_1,x')\in \Omega,\; x_1\in [0,D(\Omega)] \}).
\end{equation*}
\end{definition}
We now define two classes of domains. Let us remark that many domains share the same profile function.
Moreover, even if this profile function $g$ is concave, many of these domains may not be convex (see Figure 1).
\begin{definition}
Let $\mathcal{D}_1^{d}$ be the class defined as ($g_\Omega$ denotes the profile function of $\Omega$):
\begin{equation}
\mathcal{D}_1^d:=\{\Omega \subset \mathbb{R}^d, D(\Omega)=1 \;\mbox{and} \; g_\Omega \;\mbox{is concave }.\}
\end{equation}
More generally, we define
\begin{equation}
\mathcal{D}_\alpha^d:=\{\Omega \subset \mathbb{R}^d, D(\Omega)=1 \;\mbox{and} \; g_\Omega \;\mbox{is $1/\alpha$-concave }.\}
\end{equation}
\end{definition}
\begin{figure}[h!]\label{fig1}
\includegraphics[scale=0.18]{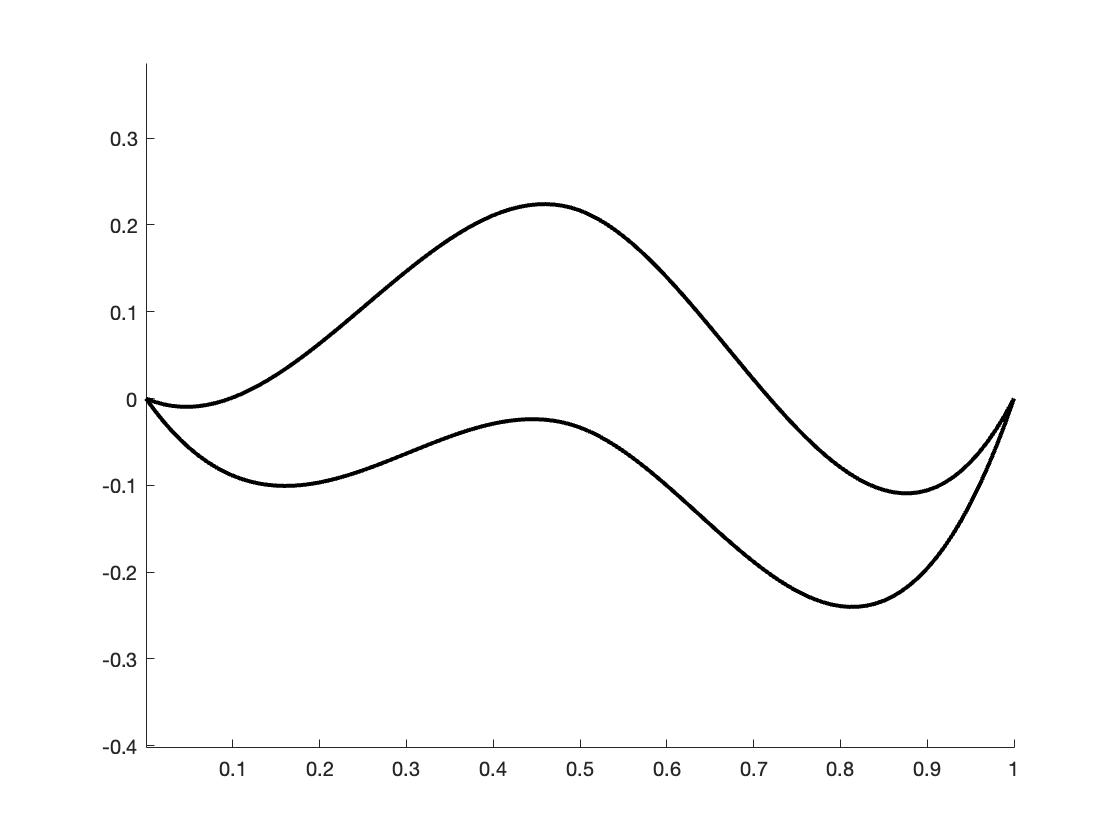}
\includegraphics[scale=0.18]{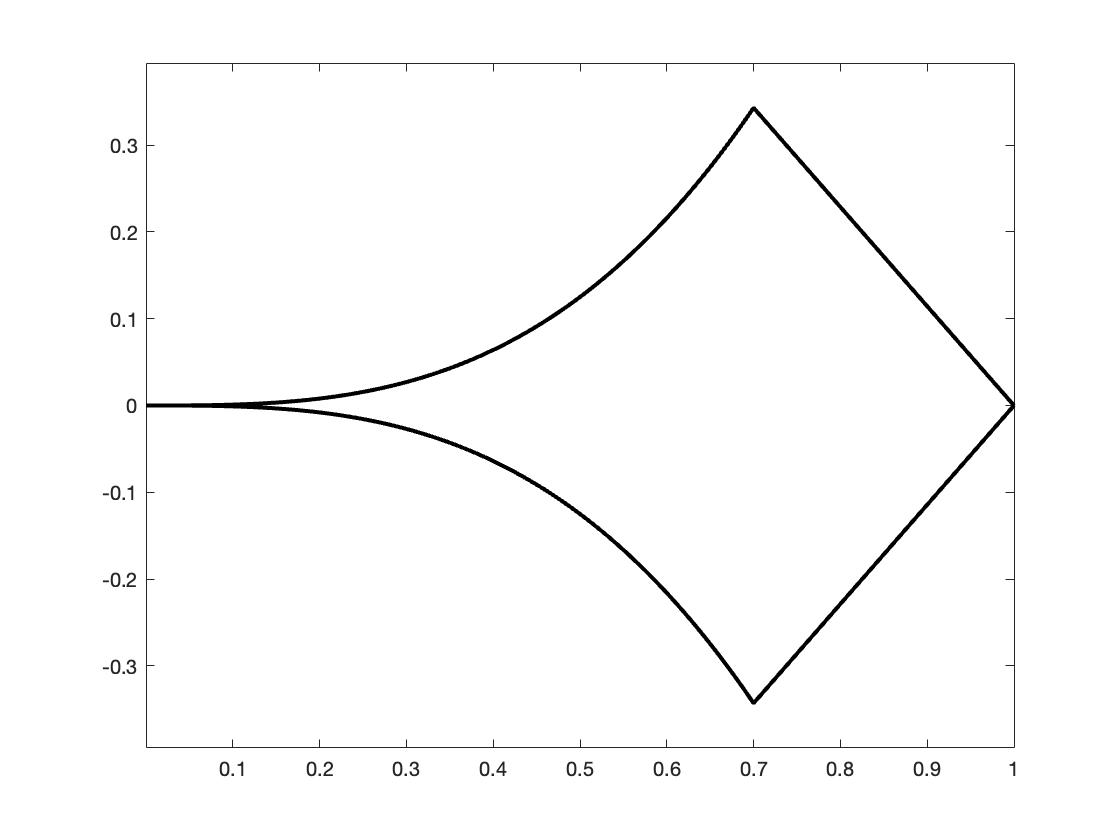}
\caption{Left: a plane domain in $\mathcal{D}_1^{2}$ (its profile function is $x(1-x)$), 
Right: a plane domain in $\mathcal{D}_3^{2}$ (its profile function is $x^3$ for $0\leq x\leq 0.7$
and affine for $x>0.7$ }
\end{figure}
Our main result is the following theorem:
\begin{theorem}\label{tmain}
Let $\Omega$ be a domain, $\Omega\in \mathcal{D}_\alpha^d$ ,(i.e. its profile function $g$ is $\frac{1}{\alpha}$-concave) with $\alpha\geq 1$, then the following bounds hold:
\begin{outline}
\1 if $\alpha<2$ then: $D(\Omega)^2\mu_k(\Omega)\leq (2j_{\frac{\alpha -1}{2}}+(k-1)\pi)^2$ 

\1 if $\alpha=2$ then: $D(\Omega)^2\mu_k(\Omega)\leq ((k+1)\pi)^2$
\1 if $\alpha>2$ then:
\2 if $k$ is odd then $D(\Omega)^2\mu_k(\Omega)\leq 4j_{\frac{\alpha-1}{2},\frac{k+1}{2}}^2$ 
\2 if $k$ is even then $D(\Omega)^2\mu_k(\Omega)\leq (j_{\frac{\alpha-1}{2},\frac{k}{2}}+j_{\frac{\alpha-1}{2},\frac{k+2}{2}})^2$ 
\end{outline}
where $j_{\nu,m}$ is the $m-$th zero of the Bessel function $J_{\nu}$. Moreover all the inequality above are optimal in the sense that they are saturated by a particular sequence of collapsing domains. 
\end{theorem}
The precise description of the collapsing sequence is given in Section \ref{section5}. 

We apply now the previous theorem to the class of convex domains.
Indeed, when $\Omega$ is convex in $\mathbb{R}^d$, a simple application of Brunn-Minkowski inequality shows
that its profile function $g$ is $(d-1)^{-1}-$concave. Therefore, the above result give a new proof of optimal upper bounds for the quantity $D(\Omega)^2\mu_k(\Omega)$ when $\Omega$ is convex. These upper bounds in the convex case reads as follows: 
\begin{corollary}[P. Kr\"oger \cite{K99}]
Let $\Omega\subset \mathbb{R}^d$ be a convex domain, then the following bounds hold:
\begin{outline}
\1 if $d=2$ then $D(\Omega)^2\mu_k(\Omega)\leq (2j_{0,1}+(k-1)\pi)^2$ 

\1 if $d=3$ then: $D(\Omega)^2\mu_k(\Omega)\leq ((k+1)\pi)^2$
\1 if $d\geq 4$ then:
\2 if $k$ is odd then $D(\Omega)^2\mu_k(\Omega)\leq 4j_{\frac{d-2}{2},\frac{k+1}{2}}^2$ 
\2 if $k$ is even then $D(\Omega)^2\mu_k(\Omega)\leq (j_{\frac{d-2}{2},\frac{k}{2}}+j_{\frac{d-2}{2},\frac{k+2}{2}})^2$ 
\end{outline}
where $j_{\nu,m}$ is the $m-$th zero of the Bessel function $J_{\nu}$. Moreover all the inequality above are optimal in the sense that they are saturated by a particular sequence of collapsing domains. 
\end{corollary}
Another thing that we want to point out is that  the geometric property of $\beta-$concavity of the profile function is crucial in order to have boundedness of the quantity $D(\Omega)^2\mu_k(\Omega)$. Indeed as soon as the profile function is $\beta-$concave then we have an optimal upper bound for $D(\Omega)^2\mu_k(\Omega)$, but this optimal bounds tends to infinity when $\beta$ goes to zero. This will be emphasized in Section \ref{section6} where we give another significant example of a sequence of domains for which 
$D(\Omega)^2\mu_k(\Omega) \to +\infty$.

A crucial role in proving the above results is played by the following relaxed version of Sturm-Liouville eigenvalues:
\begin{definition}[Relaxed Sturm-Liouville eigenvalues]\label{dMUk} 
Let $h\in L^{\infty}(0,1)$ be a non negative function, then we define the following quantity
\begin{equation}\label{vfslh}
\mu_k(h)=\inf_{E_k} \sup_{0\neq u\in E_k } \frac{\int_0^1(u')^2hdx}{\int_0^1u^2hdx},
\end{equation}
where the infimum is taken over all $k$-dimensional subspaces of the Sobolev space $H^1([0,1])$ which are  $L^2$-orthogonal to $h$ on $[0,1]$.
\end{definition} 
Note that $\mu_k(h)$ is scale invariant: for any $t>0$, $\mu_k(th)=\mu_k(h)$. As a consequence, we can if needed, fix the $L^\infty$
norm of the function $h$.

If the infimum above is achieved we have existence of eigenfunctions and the relaxed eigenvalues are actually classical eigenvalues of the following Sturm-Liouville problem
\begin{equation}\label{vFSLB}
\begin{cases}
\vspace{0.3cm}
   -\frac{d}{dx}\big(h(x)\frac{du}{dx}(x)\big)=\mu(h) h(x)u(x)  \qquad  x\in \big(0,1\big) \\
      h(0)\frac{du}{dx}(0)=h(1)\frac{du}{dx}(1)=0.
\end{cases}
\end{equation}

It is not trivial to know when the eigenfunctions exists for a relaxed Sturm-Liouville eigenvalues and also to know continuity properties of the relaxed Sturm-Liouville eigenvalues with respect to a sequence of functions $h_n$.
The existence of such eigenfunctions and the continuity of relaxed Sturm-Liouville eigenvalues are implicitly assumed in \cite{K99}.

We now explain heuristically why  the quantity $D(\Omega)^2\mu_k(\Omega)$ and the Sturm-Liouville eigenvalues are related. 
Suppose that $\Omega\subset \mathbb{R}^d$ is a domain with $D(\Omega)=1$ with profile function $h$ then, choosing as test functions  
in \eqref{eVFN}, functions $v$ only depending on the first variable $x_1$, we come out with the definition \eqref{vfslh}, and therefore
 we have that $\mu_k(\Omega)\leq \mu_k(h)$.
 Conversely,  assume that $h$ is a positive bounded function, let $\Omega_{\epsilon}$ be a sequence of domains with profile function 
 equal to $\epsilon^{d-1} h$ then we will prove in Section \ref{section2} that $\mu_k(\Omega_{\epsilon})\rightarrow \mu_k(h)$. 
Thanks to these relations we see that proving Theorem \ref{tmain} is equivalent to solve the following maximization problem:
\begin{equation*}
\sup \{\mu_k(h^{\alpha}), h:[0,1] \to [0,1], h \mbox{ concave }, \max h=1\}.
\end{equation*}
In order to give a more clear presentation of the results, we solve this problem in two steps: in Section \ref{section3} we solve the problem for $\alpha=1$ and then in Section \ref{section4} we generalize the results for $\alpha\geq 1$. 

Indeed in the case $\alpha=1$ we can use the fact that the relaxed Sturm-Liouville eigenvalues always admit an associated eigenfunction. Thanks to this fact we can write the first order optimality conditions and then prove qualitative properties of the maximizer. Showing in this way the techniques used in the proofs without being distrusted by the issue of the existence of eigenfunctions. In Section \ref{section3} we 
 use, in particular, the abstract formulation for optimality conditions already used in \cite{LaNo}, but let us note that in our case a 
 precise analysis  of the first order optimality conditions is sufficient to get the thesis (that the graph of the optimizer is polygonal).

The case $\alpha>1$ is drastically different (as we explain at the beginning of Section \ref{section4}) and in this case the existence of Sturm-Liouville eigenfunctions cannot be given for granted. For this reason in this case we first construct a maximizing sequence $h_{\epsilon}$ for which we can prove that there exists eigenfunctions associated to $\mu_k(h_{\epsilon})$. Then we study the behaviour of this maximizing sequence using the results we already proved in Section \ref{section3}    .

In Section \ref{section6} we present some examples of sequences of domains for which the quantity $D(\Omega)^2\mu_k(\Omega)$ is unbounded, another example can also be found in \cite{HKP16}. This explains why the class of domain $\Omega\subset \mathbb{R}^d$ 
for which the profile function has a $\beta-$concavity property is an optimal class of domains to obtain upper bounds for $D(\Omega)^2\mu_k(\Omega)$.

\section{Links between $\mu_k(\Omega)$ and $\mu_k(h)$}\label{section2}
In this section we study the links between the Neumann  eigenvalues and the relaxed Strum-Liouville defined in Definition \ref{dMUk}.

We prove the following bound for the supremum of the eigenvalue $\mu_1(h)$ in a particular class of functions $h$.
\begin{lemma}\label{lUBMUH}
For every $K>0$ and $A>0$, $K\leq A$, there exists a constant $C_k(K,A)$ such that the following holds 
\begin{equation*}
\sup \{\mu_k(h)\;|\;0\leq h(x) \leq A \text{ and }||h||_{L^1(0,1)}\geq K\}\leq C_k(K,A)
\end{equation*}
\end{lemma}
\begin{proof}
We introduce the following family of polynomials $\phi_1,...,\phi_k$
\begin{equation*}
\phi_i(x)=x^i-\frac{\int_0^1t^ih(t)dt}{\int_0^1h(t)dt}.
\end{equation*}
The following maximization problem:
\begin{equation*}
\sup_{v\in Span[\phi_1,...,\phi_k]}\frac{\int_0^1(v')^2hdx}{\int_0^1v^2hdx}
\end{equation*}
has a solution that we denote by $V=\sum_ia_i\phi_i$, where $a_i\in \mathbb{R}$ are constants. From Definition \ref{dMUk}, using the space $Span[\phi_1,...,\phi_k]$ as a test subspace, we obtain 
\begin{equation}\label{eUBMUH}
\mu_k(h)\leq P_k(h):= \frac{\big [\int_0^1\big (\sum_{i=1}^k ia_i x^{i-1}\big)^2h(x)dx\big]\big[\int_0^1h(x)dx \big]^2}{\int_0^1\big (\sum_{i=1}^k a_i (\int_0^1h(t)dtx^i-\int_0^1t^ih(t)dt)\big)^2h(x)dx}
\end{equation}
in order to conclude the proof we need to prove that the right hand side is bounded.

From the constraint $||h||_{L^{\infty}(0,1)}\leq A$ on the admissible functions, we can extract a maximizing sequence with the following properties 
\begin{align}\label{cCOH}
h_n &\overset{\ast}{\rightharpoonup} \overline{h}\; \text{ in } L^{\infty}(0,1)
\end{align}
and, moreover, from the constraint $||h||_{L^1(0,1)}\geq K$ we can conclude that $||\overline{h}||_{L^1(0,1)}\geq K$. It is straightforward to check that the functional $P_k(h)$ is continuous under the convergence \eqref{cCOH}. From this fact and from \eqref{eUBMUH} we finally obtain
\begin{equation*}
\mu_k(h)\leq P_k(\overline{h})=C_k(K,A)
\end{equation*} 
\end{proof}
In the next two lemmas we explain the relations between the first Neumann eigenvalue and the quantity $\mu_1(h)$. In the first lemma we 
prove that if $\Omega\subset \mathbb{R}^d$ and $h$ is its profile function (see Definition \ref{dprofile}) $D(\Omega)^2\mu_k(\Omega)\leq \mu_k(h)$. In the second lemma we prove that if $\Omega_{\epsilon}$ is a sequence of domains with profile function 
equal to $\epsilon^{d-1} h$ then $D(\Omega_{\epsilon})^2\mu_k(\Omega_{\epsilon})\rightarrow \mu_k(h)$. 
\begin{lemma}\label{lUBMU}
Let $\Omega$ be an open and Lipschitz domain and $h$ its profile function in the sense of Definition \ref{dprofile}, then 
\begin{equation*}
D(\Omega)^2\mu_k(\Omega)\leq \mu_k(h)
\end{equation*}
\end{lemma}
\begin{proof}
Without loss of generality we can assume $D(\Omega)=1$. From the variational characterization \eqref{eVFN}, using test functions 
that depend only on $x_1$, we recover the definition \eqref{vfslh} of the relaxed Sturm-Liouville eigenvalue $\mu_k(h)$.
\end{proof}
Now we are interested in the behavior of the Neumann eigenvalues on collapsing domains: 
\begin{lemma}\label{lLBMU}
Let $h\in L^{\infty}(0,1)$ be a
non negative function (not identically zero), let $\Omega_{\epsilon, h}$ be a domain with profile function  equal to $\epsilon^{d-1} h$ then 
\begin{equation*}
\lim_{\epsilon\to 0}D(\Omega_{\epsilon, h})^2\mu_k(\Omega_{\epsilon, h})=\mu_k(h).
\end{equation*}  
\end{lemma}
\begin{proof}
Without loss of generality we can assume that $D(\Omega_{\epsilon, h})\rightarrow 1$. According to Lemma \ref{lUBMU} we have that:
\begin{equation}\label{elimsup}
\mu_k(\Omega_{\epsilon h})\leq \mu_k(\epsilon^{d-1}h)=\mu_k(h),
\end{equation}
the last equality is true because from Definition \ref{dMUk} the eigenvalue $\mu_k(h)$ is invariant under multiplication  of $h$ with a positive constant. In particular from Lemma \ref{lUBMUH} we conclude that there exists a constant $C_k(h)$, that depends on $h$ but does not depend on $\epsilon$, such that 
\begin{equation*}
\mu_k(\Omega_{\epsilon, h})\leq C_k(h).
\end{equation*} 
Let $u_{k,\epsilon}$ be the eigenfunction associated to the eigenvalue $\mu_k(\Omega_{\epsilon, h})$, normalized in such a way that $||u_{k,\epsilon}||_{L^2(\Omega_{\epsilon, h})}=1$, we introduce the following functions 
\begin{equation*}
\overline{u}_{k,\epsilon}(x_1,x')=\epsilon^{\frac{d-1}{2}}u_{k,\epsilon}(x_1,\epsilon x')\quad \forall\; (x_1,x')\in \Omega_h,
\end{equation*}
where $\Omega_h$ is a domain with profile function equal to $h$. We want to prove that the sequence of functions $\overline{u}_{k,\epsilon}$ is bounded in $H^1(\Omega_h)$. We start by bound the quantity $||\nabla \overline{u}_{k,\epsilon}||_{L^2(\Omega_h)}$
\begin{equation*}
\int_{\Omega_h}|\nabla \overline{u}_{k,\epsilon}|^2dx=\epsilon^{d-1}\int_{\Omega_h}(\frac{\partial u_{k,\epsilon}}{\partial x_1})^2+\epsilon^2 |\nabla_{x'} u_{k,\epsilon}|^2dx\leq \int_{\Omega_{\epsilon, h}}|\nabla u_{k,\epsilon}|^2dy=\mu_k(\Omega_{\epsilon, h})\leq   C_k(h),
\end{equation*}
in the inequalities above we use the change of variables $y_1=x_1$ and $y'=\epsilon x'$. Using the same change of variable we also conclude that $||\overline{u}_{k,\epsilon}||_{L^2(\Omega_h)}=||u_{k,\epsilon}||_{L^2(\Omega_{\epsilon, h})}=1$.

We conclude that there exists a function $\overline{U}_k\in H^1(\Omega_h)$ such that, up to a subsequence,  
\begin{align}\label{cCNE}
\overline{u}_{k,\epsilon} &\rightharpoonup \overline{U}_k\; \text{ in } H^1(\Omega_h)\\ \notag
\overline{u}_{k,\epsilon} &\rightarrow \overline{U}_k\; \text{ in } L^2(\Omega_h).
\end{align}
We now prove that the limit function $\overline{U}_k$ does not depend on $x_i$ where $i=2,...,d$, indeed we have that 
\begin{equation*}
\int_{\Omega_h}(\frac{\partial \overline{U}_k}{\partial x_i})^2dx\leq \liminf \int_{\Omega_h}(\frac{\partial \overline{u}_{k,\epsilon}}{\partial x_i})^2dx=\liminf \epsilon^2 \int_{\Omega_{\epsilon, h}}(\frac{\partial u_{k,\epsilon}}{\partial x_i})^2dx=0.
\end{equation*}
We denote by $U_k$ the restriction of the function $\overline{U}_k$ to the variable $x_1$. We can do the same argument with all the other eigenfunctions $u_i$. From the convergences given in \eqref{cCNE} we conclude that, for $\epsilon$ small enough we have that 
\begin{equation}\label{eLBNE}
\mu_k(\Omega_{\epsilon, h})=\max_{\beta\in \mathbb{R}^k}\frac{\sum_i \beta_i^2\int_{\Omega_{\epsilon, h}}|\nabla u_{i,\epsilon}|^2dx}{\sum_i \beta_i^2\int_{\Omega_{\epsilon h}} u_{\epsilon}^2dx}\geq\max_{\beta\in \mathbb{R}^k} \frac{\sum_i \beta_i^2\int_0^1 {U_i'}^2 hdx}{\sum_i \beta_i^2\int_0^1 U_i^2 hdx}+f(\epsilon,h)
\end{equation}
where $f(\epsilon,h)$ goes to 0 when $\epsilon \to 0$.
From the fact that $u_{\epsilon}$ is a Neumann eigenfunction and from the convergences \eqref{cCNE} it is clear that $\int_0^1U_ih=0$, from Definition \ref{dMUk} and from \eqref{eLBNE}, using the test subspace given by $Span[U_1,...,U_k]$ for $\mu_k(h)$ we obtain that
\begin{equation*}
\liminf_{\epsilon\to 0}\mu_k(\Omega_{\epsilon, h})\geq \mu_k(h).
\end{equation*}
Combining this inequality with \eqref{elimsup} we conclude that $\lim D(\Omega_{\epsilon, h})^2\mu_k(\Omega_{\epsilon, h})=\mu_k(h)$.
\end{proof}
From the two previous lemmas, we deduce
\begin{proposition}\label{equalsup}
We have equality between the two quantities
$$\begin{array}{c} \vspace{3mm}
S_1:=\sup \{D(\Omega)^2\mu_k(\Omega), \Omega \subset \mathbb{R}^d, \mbox{ bounded, Lipschitz } \} \\
S_2:=\sup \{\mu_k(h), h\in L^\infty(0,1), h\geq 0,  h\not= 0 \}.
\end{array}$$
\end{proposition}
Let $\alpha\geq 1$,therefore we are interested in solving the following problem:
\begin{equation*}
\sup \{\mu_k(h^{\alpha}), h\in \mathcal{L}\}.
\end{equation*}
where:
\begin{equation*}
\mathcal{L}:=\{h:[0,1] \to [0,1], h \mbox{ concave }, \max h=1\}.
\end{equation*}
In this way, we will get sharp upper bounds for $D(\Omega)^2\mu_k(\Omega)$ when $\Omega\subset \mathbb{R}^d$ is a set with a profile function that is $\frac{1}{\alpha}-$concave.
\section{Maximization of $\mu_k(h)$}\label{section3}
In this section we study the following maximization problem:
\begin{equation*}
\sup \{\mu_k(h), h\in \mathcal{L}\}. 
\end{equation*}
We first study the particular case $\alpha=1$ because in this case we do not have issues related to the existence of eigenfunctions of 
relaxed Strum-Liouville eigenvalues. In this way the main ideas concerning the use of the first order optimality conditions and the 
proofs of the qualitative property of the maximizers  will be more clear. 

Let us start by the following remark (see \cite{HM21}, \cite{T65}) that follows from the explicit expression of the Green kernel
of the corresponding ODE that turns out to be in $L^2((0,1)\times (0,1))$.
\begin{remark}\label{rmkwelldefined}
Let us assume that there exist $K>0$ and $p<2$ such that $h(x)\geq K(x(1-x))^p$ a. e. in $(0,1)$, then $\mu_k(h)$ defined in  Definition \ref{dMUk} is the $k-$th eigenvalue of the following eigenvalue problem.
\begin{equation*}
\begin{cases}
\vspace{0.3cm}
   -\frac{d}{dx}\big(h(x)\frac{du}{dx}(x)\big)=\mu(h) h(x)u(x)  \qquad  x\in \big(0,1\big) \\
      h(0)\frac{du}{dx}(0)=h(1)\frac{du}{dx}(1)=0,
\end{cases}
\end{equation*}
This is the case, in particular, for a concave (positive inside) function $h$ since such a function satisfies $h(x)\geq K(x(1-x))$ for a positive $K$.
\end{remark}
We recall two important properties for the class $\mathcal{L}$ and the eigenvalues $\mu_k(h)$, see Lemma 3.1 and Lemma 3.6 in \cite{HM21}.
\begin{proposition}\label{propcoco}
\begin{itemize}
\item For any sequence $h_n$ in $\mathcal{L}$, there exists a subsequence (still denoted $h_n$) and a function $h\in \mathcal{L}$ such that
$h_n$ converge to $h$ in $L^2(0,1)$ and uniformly on every compact subset of $(0,1)$.
\item If $h_n \to h$ in $L^2(0,1)$ then, for any $k$, $\mu_k(h_n) \to \mu_k(h)$.
\end{itemize}
\end{proposition}
\begin{proof}
See Lemma 3.1 and Lemma 3.6 in \cite{HM21} for the proof. We want only to stress the fact that the second point strongly relies on the fact that the Green kernel $g(h;x,y)$ associated to the Sturm-Liouville problem under consideration is in $L^2(0,1)times (0,1)$
\end{proof}

The main result of this section is the following theorem. We recall that $j_{0,1}$ is the first zero of the Bessel function $J_0$.
\begin{theorem}\label{tEFHO}
For any $k\geq 1$, the problem $\max \{\mu_k(h), h\in \mathcal{L}\}$ has a solution $h_k^*$, moreover 
\begin{itemize}
\item $\max \{\mu_1(h), h\in \mathcal{L}\}=\mu_1(h_1^*)=(2j_{0,1})^2$ and 
\begin{equation*}
h^*_1=\begin{cases}
2x \quad &x\in [0,\frac{1}{2}], \\
2(1-x) \quad &x\in [\frac{1}{2},1]
\end{cases}
\end{equation*}
\item let $k\geq 2$ then $\max \{\mu_k(h), h\in \mathcal{L}\}=\mu_k(h_k^*)=(2j_{0,1}+(k-1)\pi)^2$ 
\begin{equation*}
h^*_k=\begin{cases}
\frac{x(2j_{0,1}+(k-1)\pi)}{j_{0,1}} \quad &x\in [0,\frac{j_{0,1}}{(2j_{0,1}+(k-1)\pi)}] \\
1 \quad &x\in [\frac{j_{0,1}}{(2j_{0,1}+(k-1)\pi)},1-\frac{j_{0,1}}{(2j_{0,1}+(k-1)\pi)}]\\
\frac{(1-x)(2j_{0,1}+(k-1)\pi)}{j_{0,1}} \quad &x\in [1-\frac{j_{0,1}}{(2j_{0,1}+(k-1)\pi)},1]
\end{cases},
\end{equation*}
\end{itemize}
\end{theorem}
We prove this theorem in several steps, the first important one in order to prove this result is the following theorem

\begin{theorem}\label{theok+1}
For any $k\geq 1$, the problem $\max \{\mu_k(h), h\in \mathcal{L}\}$ has a solution. This one has a graph that is a polygonal line composed of (at most) $k+1$ segments.
\end{theorem}

The existence of a maximizer follows immediately from the compactness of the class $\mathcal{L}$ and the continuity of the eigenvalues
stated in Proposition \ref{propcoco}. We will denote by $h_k^*$ (or simply $h^*$ if no confusion can occur)  a maximizer.

\medskip
For the qualitative properties of the maximizer, we will mainly use the optimality conditions  but we have to take into account the concavity constraint
and the bound constraints $0\leq h(x) \leq 1$. Let us first give the derivative of the eigenvalue.
We consider $t>0$ a positive number, and we define the following derivatives:
for every $\phi\in \mathcal{L}$ we define $\mu_{t,\phi}:=\mu_k(h+t\phi)$ and we denote by $u_{t,\phi}$ the corresponding eigenfunction. 
We use the following notation for the derivative of the eigenvalue (since this is a Sturm-Liouville problem, we know that the eigenvalue is simple
and then differentiable):
\begin{equation*}
\dot{\mu}_{\phi}:=\frac{d}{dt}\mu_k(h+t\phi)\Big |_{t=0}
\end{equation*} 
and the following notation for the derivative of the eigenfunction :
\begin{equation*}
\dot{u}_{\phi}:=\frac{d}{dt}u_{t,\phi}\Big |_{t=0}.
\end{equation*} 
\begin{lemma}\label{ldermuk}
The derivative of $\mu_k$ in the direction $\phi$ is given by
\begin{equation}\label{dermuk}
\dot{\mu}_{\phi}=\int_0^1 ({u'}^2 - \mu_k u^2) \phi dx
\end{equation}
where $u$ is the eigenfunction associated to $\mu_k(h)$ normalized by $\int_0^1 hu^2 =1$.
\end{lemma}
\begin{proof}
Since this kind of perturbation is classical,  see e.g. \cite[section 5.7]{HPb} we just perform a formal computation here, the complete justification would involve
an implicit function theorem together with Fredholm alternative.
Let us compute $\dot{\mu}_{\phi}$, from \eqref{vFSLB} we know that 
\begin{equation*}
\frac{d}{dt}\Big [ -\frac{d}{dx}\big((h+t\phi)\frac{d u_{t,\phi}}{dx}\big)\Big ]\Big |_{t=0}=\frac{d}{dt}[\mu_{t,\phi}(h+t\phi) u_{t,\phi}]\Big |_{t=0},
\end{equation*}
so we obtain the following differential equation satisfied by $\dot{u}_{\phi}$
\begin{equation}\label{eUD}
-\frac{d}{dx}\left(h\frac{d \dot{u}_{\phi}}{dx}\right)-\frac{d}{dx}\left(\phi\frac{d u}{dx}\right)=\dot{\mu}_{\phi} h u +\mu_k h \dot{u}_{\phi} + \mu_k \phi u.
\end{equation}
Multiplying both side of the above equation by $u$ and integrating we obtain 
\begin{equation}\label{eND}
\int_0^1 h u' {\dot{u}_{\phi}}'-uh{\dot{u}_{\phi}}'\|_{x=0}^{x=1} + \int_0^1 \phi {u'}^2 - u\phi u''\|_{x=0}^{x=1} = \dot{\mu}_{\phi} + \mu_k \int_0^1 \phi u^2 +\mu_k
\int_0^1 h u \dot{u}_{\phi}
\end{equation}
where we have used $\int_0^1 hu^2 =1$.
Now from the variational formulation satisfied by $u$ we see that
$$\int_0^1 h u' {\dot{u}_{\phi}}' = \mu_k \int_0^1 h u \dot{u}_{\phi}$$
On the other hand, differentiating w.r.t. $t$ the boundary condition $(h+t\phi)(0)u_{t,\phi}(0)$ yields
$$\phi(0) u(0) +h(0) \dot{u}_{\phi}(0) =0 \quad \mbox{and}\quad \phi(1) u(1) +h(1) \dot{u}_{\phi}(1) =0$$
therefore the boundary terms cancel and formula \eqref{dermuk} follows.
\end{proof}
Let us now write the first order optimality condition, taking into account the concavity constraint (i.e. $h^{\prime\prime} \leq 0$) and the bound constraints
$0\leq h\leq 1$. For that purpose, we follow the formalism of \cite{LaNo}, (Proposition 2.3.2) see also the abstract framework in \cite{MaZo}:
\begin{proposition}\label{propoptco}
Let $h^*$ be a maximizer of $\mu_k(h)$ in the class $\mathcal{L}$,. We denote by $S$ the support of (the negative measure) $h^{\prime\prime}$. Then there exist
\begin{itemize}
\item a function $\xi \in H^1(0,1)$, such that $\xi \geq 0$, $\xi=0$ on $S$,
\item two non-negative Radon measures $\nu_0,\nu_1$ such that $suppt(\nu_0)\subset \{x|h^*(x)=0\}$ and $suppt(\nu_1)\subset \{x|h^*(x)=1\}$
\end{itemize}
such that, for any $\phi \in H^1(0,1)$
\begin{equation}\label{condopt}
<\dot{\mu}_{\phi},\phi>=-<\xi^{\prime\prime},\phi> +\int_0^1 \phi d\nu_0 - \int_0^1 \phi d\nu_1 .
\end{equation}
\end{proposition}
Note that, since $h^*$ is non-negative and concave with $\max h^*=1$ we have the following properties of the sets $\{x/h^*(x)=0\}$ and $\{x/h^*(x)=1\}$:
\begin{equation}\label{h=0}
\{x/h^*(x)=0\} \subset \{0\}\cup \{1\}
\end{equation}
\begin{equation}\label{h=1}
\{x/h^*(x)=1\} \;\mbox{is empty or an interval } [a_1,a_2] (\mbox{ with $0<a_1\leq a_2<1$}).
\end{equation}
Moreover, since the eigenvalue $\mu_k(h)$ is invariant when we multiply $h$ by a positive constant, we can even assume, if needed, that the bound constraint
$h^*=1$ is not saturated.

Let us introduce the function $f:={u'}^2-\mu_k u^2$, then, according to \eqref{dermuk}, $<\dot{\mu}_{\phi},\phi>=\int_0^1 f\phi dx$.
From the ODE satisfied by the eigenfunction, we see that $u\in H^2(0,1)\subset C^1([0,1])$, therefore $f$ is continuous on $[0,1]$.

Taking $\phi$ compactly supported in $(0,a_1)\cup (a_2,1)$ in \eqref{condopt} yields
$$\int_0^1 f\phi dx = -<\xi^{\prime\prime},\phi>$$
therefore, $\xi$ satisfies, in the sense of distributions
\begin{equation}\label{eqxi}
- \xi^{\prime\prime} = f \quad \mbox{on } (0,a_1)\cup (a_2,1).
\end{equation}
We also know that $\xi$ vanishes on  $S$ the support of ${h^*}^{\prime\prime}$. This support being closed, its complement
$S^c$ is a union of intervals: $S^c=\bigcup_{i\in I} (\alpha_i,\beta_i)$. 
We will distinguish the internal intervals (those for which $\alpha_i>0$ and $\beta_i<1$) and the boundary intervals: it will be a consequence
of the proof below that the number of internal intervals is finite, thus we will have only two boundary intervals that we will denote $(0,\beta_b)$
and $(\alpha_b,1)$.

Let us start with the internal intervals, on such an interval $(\alpha_i,\beta_i)$, $\xi$ is solution of
\begin{equation}\label{eqxi2}
\left\lbrace\begin{array}{l}
- \xi^{\prime\prime} = f \quad \mbox{on } (\alpha_i,\beta_i) \\
\xi(\alpha_i)=\xi(\beta_i) =0
\end{array}\right.
\end{equation}
We infer:
\begin{lemma}\label{lemmaxi}
The function $f$ vanishes at least two times on each internal interval $(\alpha_i,\beta_i)$.
\end{lemma}
\begin{proof}
Let us denote $(\alpha,\beta)=(\alpha_i,\beta_i)$.
From equation \eqref{eqxi2} we can get an explicit expression of $\xi$ on the interval $(\alpha,\beta)$, namely
\begin{equation}\label{explixi}
\xi(x)=\frac{x-\alpha}{\beta - \alpha} \int_\alpha^\beta (\beta-t) f(t) dt - \int_\alpha^x (x-t) f(t) dt  .
\end{equation}
Now, since $f$ is continuous, the function $\xi$ belongs to $H^2((0,a_1)\cup (a_2,1))$ and then it is $C^1$ (and even $C^2$).
But $\xi$ being non-negative on $(0,a_1)\cup (a_2,1)$, we necessarily have $\xi'(\alpha)=\xi'(\beta)=0$. With the explicit expression
\eqref{explixi}, we get $\xi'(x)=\frac{1}{\beta - \alpha} \int_\alpha^\beta (\beta-t) f(t) dt - \int_\alpha^x  f(t) dt$
thus, this provides the two relations
$$\int_\alpha^\beta (\beta-t) f(t) dt=0 \quad \int_\alpha^\beta  f(t) dt=0$$
or
\begin{equation}\label{eOCf}
\int_\alpha^\beta f(t) dt=0 \quad \int_\alpha^\beta  t f(t) dt=0.
\end{equation}

These two relations imply the thesis. Indeed from the first one $f$ must vanish at least one time, say at $x=\gamma$. Assuming,
that $f$ vanishes only at $\gamma$ would provide $\int_\alpha^\beta (t-\gamma) f(t) dt \not= 0$ a contradiction.
\end{proof}
In order to get the result announced in Theorem \ref{theok+1}, we need to prove that the complement $S^c$ of the support of 
${h^*}^{\prime\prime}$ has only $k+1$ components (including possibly the interval $(a_1,a_2)$ if it is not empty and two boundary intervals). For that purpose,
we study the function $f$ on any nodal domain of $u$.  We distinguish here the "internal" nodal intervals, those where $u$ vanishes at the two extremities
and the two "boundary" nodal intervals (where $u$ vanishes a priori only at one extremity).
The key proposition is the following:
\begin{proposition}\label{propkey}
The function $f$ vanishes exactly two times on an internal nodal interval of $u$ and exactly one time on a boundary nodal interval.
\end{proposition}
Since $u$ is the $k+1$-th eigenfunction of a Sturm-Liouville problem, its number of nodal domains is exactly $k+1$ ($k-1$ internal plus two boundary nodal intervals).
Following Proposition \ref{propkey}, it implies that the function $f$ has exactly $2k$ zeros. The desired result will follow easily, see below.

\begin{proof}[Proof of Proposition \ref{propkey}]
In the proof, the optimal function $h^*$ will be simply denoted $h$ and the eigenvalue is $\mu$.
Let us consider first an internal nodal interval $[a,b]$ of the eigenfunction $u$. Without loss of generality, we can assume that $u>0$ on $(a,b)$
and we also assume that $(a,b)\subset (0,a_1)$ (i.e. this nodal interval is before the maximum of $h^*$ (where possibly $h^*=1$). For the other case,
$(a,b)\subset [a_2,1)$, the proof will follow the same lines. Therefore, on this interval $h^\prime$ is decreasing (by concavity)
and $h$ is increasing, in particular, $h^\prime/h > 0$. The differential equation satisfied by $u$ can be written
\begin{equation}\label{equau}
u^{\prime\prime}=-\frac{h^\prime}{h} u^\prime -\mu u.
\end{equation}
This shows that $u$ is $W^{2,\infty}_{loc}$ in $(0,1)$. More precisely, $u^{\prime\prime}$ can be discontinuous at discontinuity points of $h^\prime$
(except if $u^\prime$ vanishes at these points).

{\bf 1st step:} $u$ is unimodal on $(a,b)$ in the sense that $u$ is increasing on some interval $(a,c)$ and then decreasing on $(c,b)$.\\
Let us consider a point $c$ where $u$ achieves its maximum on $[a,b]$. We prove now that $u$ is decreasing after $c$.
Let us assume, for a contradiction, that there is a point $x_0\geq c$ with $u^\prime(x_0)>0$. Let us introduce the point
$$x_1:=\max\{x\in [a,x_0], u^\prime(x)\leq 0\}.$$
By continuity, we have $u^\prime(x_1)=0$ and by definition of $x_1$, we have $u^\prime>0$ on $(x_1,x_0]$.
Therefore, using the equation \eqref{equau}, we see that $u^{\prime\prime} <0$ on $[x_1,x_0]$ yielding 
$0<u^\prime(x_0)\leq u^\prime(x_1)=0$ a contradiction.

In the same way, we prove by contradiction that $u$ is increasing on $[a,c]$ assuming that there exists $x_0<c$ such that
$u^\prime(x_0)<0$ and introducing
$$x_1:=\max\{x\in [x_0,c], u^\prime(x)< 0\}$$
we get $u^\prime(x_1)=0$, $u^\prime \geq 0$ on $[x_1,c]$, therefore $u^{\prime\prime} <0$ on $(x_1,c)$ providing a contradiction
since $u^\prime(c)=u^\prime(x_1)=0$.

Another way to express these properties is the following: when $u>0$, the only points where $u^\prime$ vanishes correspond to (local) maxima,
since, by the equation $u^{\prime\prime} <0$, there.

\medskip
Now let us remark that $f(a)={u^\prime(a)}^2\geq 0$, $f(c)=-\mu u^2(c)<0$ and $f(b)={u^\prime(b)}^2\geq 0$, thus $f$ vanishes
at least once between $a$ and $c$ and at least once between $c$ and $b$. Thus we need to prove:\\
{\bf 2nd step:} The function $f$ cannot vanish more
than one time on each interval $[a,c]$ and $[c,b]$. \\
On the first interval $[a,c]$ it is clear: since $u^\prime > 0$ there, we have by \eqref{equau} $u^{\prime\prime} <0$ on $(a,c)$.
Now $f^\prime$ being given by
\begin{equation}\label{fprime}
f^\prime=2u^\prime(u^{\prime\prime}-\mu u)=-2u^\prime(\frac{h^\prime}{h} u^\prime +2\mu u)
\end{equation}
we see that $f^\prime < 0$ on $(a,c)$ and therefore $f$ vanishes only once.

Now we look at the interval $(c,b)$. Since $f(c)<0$ and $f(b)>0$ $f$ vanishes an odd number of time between $c$ and $b$.
Assume, for a contradiction, that $f$ vanishes at (least) three times and take three consecutive zeros $x_1,x_2,x_3$. Writing
$f=(u^\prime+\sqrt{\mu} u)( u^\prime-\sqrt{\mu} u)$ and observing that $u^\prime-\sqrt{\mu} u <0$ on $(c,b)$, we infer that the function
$g=u^\prime+\sqrt{\mu} u$ vanishes at the same points $x_1,x_2,x_3$.
Let us also observe that, since $h$ is positive and increasing while $h^\prime$ is decreasing by concavity, the function $h^\prime/h$ 
is decreasing on $(a,b)$.\\
Let us now introduce the function $\varphi:=u/u^\prime$. By definition, $\varphi <0$ on the interval $(c,b)$. The differential equation satisfied by $\varphi$ is 
\begin{equation}\label{eqvarphi}
\varphi^\prime=1-\frac{u u^{\prime\prime}}{{u^\prime}^2}=1+\frac{h^\prime}{h}\,\varphi +\mu \varphi^2.
\end{equation}
Now, $g(x_i)=0$ implies $\varphi(x_1)=\varphi(x_2)=\varphi(x_3)=-1/\sqrt{\mu}$.  Therefore, there exist two points $y_1 \in (x_1,x_2)$ and $y_2\in (x_2,x_3)$
such that $\varphi^\prime(y_1)=\varphi^\prime(y_2)=0$. Moreover, we can assume that $|\varphi| < 1/\sqrt{\mu}$ on $(x_1,x_2)$
(and $|\varphi| > 1/\sqrt{\mu}$ on $(x_2,x_3)$).
Coming back to \eqref{eqvarphi}, we see that $\varphi(y_1)$ and $\varphi(y_2)$ must be solutions of the quadratic
$$1+\frac{h^\prime}{h}\,X +\mu X^2 =0.$$
This implies that 
$$\frac{h^\prime}{h}\,\left(y_1\right) \geq \frac{h^\prime}{h}\,\left(y_2\right) \geq 2\sqrt{\mu}.$$
Now, let us introduce $z_1\in (x_1,x_2)$ a maximum of $f$ on this interval. From $f^\prime(z_1)=0$, we infer 
$u^{\prime\prime}(z_1)= \mu u(z_1)$ and therefore
$$\varphi^\prime(z_1)= 1-\frac{\mu u^2(z_1)}{{u^\prime}(z_1)^2}=1-\mu \varphi(z_1)^2$$
that yields with \eqref{eqvarphi}
\begin{equation}\label{eqz1}
\frac{h^\prime}{h}\,\left(z_1\right) = -2\mu \varphi(z_1).
\end{equation}
More precisely, we see that 
$$f^\prime \leq 0\ \Leftrightarrow\  \frac{h^\prime}{h}\,\geq -2\mu \varphi$$
now, by \eqref{eqz1} and the fact that $|\varphi|< 1/\sqrt{\mu}$ on $(x_1,x_2)$, we see that $f^\prime(y_1)\leq 0$ thus $y_1\geq z_1$.
We use now the fact that $h^\prime/h$ is decreasing to write
$$2\sqrt{\mu} < \frac{h^\prime}{h}\,\left(y_1\right)  \leq \frac{h^\prime}{h}\,\left(z_1\right) =-2\mu \varphi(z_1) < 2\mu \frac{1}{\sqrt{\mu}}$$
yielding a contradiction.

It remains to look at a nodal interval of the kind $I=[0,a_1]$. For that purpose, we need the following lemma.
\begin{lemma}\label{lhvanish}
let $h_k^*$ be a maximizer for the problem $\max \{\mu_k(h), h\in \mathcal{L}\}$ then $h_k^*(0)=h_k^*(1)=0$. 
\end{lemma}
\begin{proof}
In the proof, the optimal function $h_k^*$ will be simply denoted by $h$, the eigenvalue is $\mu$ and the eigenfunction is $u$. Suppose by contradiction that  $h(0)>0$, let $\eta \in (0,1)$ then we define the following function 
\begin{equation*}
\phi_{\eta}=\begin{cases}
x-\eta \quad x\in [0,\eta) \\
0 \quad x\in [\eta ,1]. 
\end{cases}
\end{equation*}
for $\eta$ small enough we have that the function $h+t\phi_{\eta}$ is in $\mathcal{L}$ for $t$ small enough (we recall that we can assume $h<1$ because $\mu(h)$ is invariant by scaling) in particular we can compute the derivative of the eigenvalue, and we obtain that 
\begin{equation}
\dot{\mu}_{\phi_{\eta}}=\int_0^1 ({u'}^2 - \mu_k u^2)\phi_{\eta} dx=0.
\end{equation}
From the fact that $h(0)>0$ and the boundary conditions we know that $u'(0)=0$. Now assume that $u(0)\neq 0$, so there exists $\eta_1$ for which ${u'}^2 - \mu_k u^2<0$ in $[0,\eta_1]$ and this will imply that 
  \begin{equation}
\dot{\mu}_{\phi_{\eta_1}}=\int_0^1 ({u'}^2 - \mu_k u^2)\phi_{\eta_1} dx>0,
\end{equation}
that is a contradiction, so we conclude that $u(0)=0$. In particular the eigenfunction $u$ must solve the following Cauchy problem 
\begin{equation*}
\begin{cases}
u^{\prime\prime}+\frac{h^\prime}{h} u^\prime +\mu u=0 \quad x\in [0,\eta_1) \\
u(0)=0 \\
u'(0)=0.
\end{cases}
\end{equation*}
this will imply that $u\equiv 0$ in an interval of positive measure, that is a contradiction. We can adapt the same argument in order to prove that $h(1)=0$.
\end{proof}
We come back to the analysis of $f={u^{\prime}}^2-\mu u^2$ on the first nodal interval $[0,a_1]$. Since $h(0)=0$, we can prove that $u^\prime(0)=0$.
Indeed, using the representation of the eigenfunction given in \cite{T65}, namely
$$u(x)=\mu \int_0^1 g(x,y) h(y) u(y) dy\quad \mbox{ where } g(x,y)=\int_0^{\min(x,y)} \frac{t\,dt}{h(t)} +\int_{\max(x,y)}^1 \frac{(1-t)\,dt}{h(t)}$$
we get by differentiation
$$u^\prime(x) = \mu\left(\frac{x}{h(x)} \int_x^1 h(y) u(y) dy - \frac{1-x}{h(x)} \int_0^x h(y) u(y) dy\right).$$
Since $h$ is concave (non identically zero) and $h(0)=0$, we have $h(x)\geq c x$ for some positive $c$ in a neighborhood of $0$, therefore passing to the limit in the expression above, and using
$\int_0^1 h(y) u(y) dy=0$, we get $u^\prime(0)=0$. We infer $f(0)<0$, $f(a_1)>0$ and we can follow exactly the proof of the general case to obtain that $f$ vanishes
only once in the interval $[0,a_1]$. The same holds for the last nodal interval.

\end{proof}
We are now ready to prove Theorem \ref{theok+1}
\begin{proof}[Proof of Theorem \ref{theok+1}]
we have seen that the function $f$ vanishes exactly $2k$ times on $(0,1)$. Moreover, according to Lemma \ref{lemmaxi}, $f$ vanishes at least two times on each internal 
interval contained in the complement of the support of the measure $h^{\prime\prime}$. Therefore, there are at most $k$ such internal intervals 
denoted $(\alpha_i,\beta_i), i=1,\ldots M-1$, $M\leq k+1$ and possibly two boundary
intervals $(0,x_1)$ and $(x_{M},1)$. Let us first prove that the support $S$ is discrete and contains only the extremities of the previous intervals,
namely the $x_i$. Assume, for a contradiction that, for some $i$,  $\beta_i< \alpha_{i+1}$. In that case, the interval $J=[\beta_i, \alpha_{i+1}]$ is contained in the support $S$
and, according to the optimality condition given in Proposition \ref{propoptco}, we have 
$${u^\prime}^2-\mu u^2 =0 \;\mbox{on } J \Longrightarrow u^\prime = \pm \sqrt{\mu} u  \;\mbox{on } J.$$
This implies immediately that $u^{\prime\prime}=\mu u$ on $J$ and coming back to the equation satisfied by the eigenfunction, we infer
$h^\prime/h=\pm 2\sqrt{\mu}$ (on a subinterval $J'$ where $u$ does not vanish), or $h(x)=\exp(2\sqrt{\mu} x)$ on $J'$,
 but this is impossible since we assumed $h$ to be concave. We have proved that the support of $h^{\prime\prime}$ is the discrete set
 $S=\{0,x_1,\ldots, x_M,1\}$ and then $h$ is affine on each interval $(x_i,x_{i+1})$ (we set $x_0=0$ and
 $x_{M+1}=1$). 
 
 Now we are going to prove that the function $f$ vanishes at least once on each boundary interval $(0,x_1)$
 and $(x_M,1)$. For that purpose, we still use the optimality condition given in Proposition \ref{propoptco} on this interval.
 We have now to take into account the non-negative Radon measure $\nu_0$. Since its support is restricted to the two points $0$ and $1$
 (as a consequence of Lemma \ref{lhvanish}), it can be written as a sum of two Dirac measures: $\nu_0=t_0 \delta_0 + t_1 \delta_1$
 with $t_0\geq 0$ and $t_1\geq 0$. Now coming back to the equation \eqref{condopt} satisfied by $\xi$, we get for any $\phi \in H^1(0,1)$:
 $$\int_0^1 f\phi =\int_0^1 \xi^\prime\phi^\prime+t_0\phi(0) +t_1\phi(1)\,.$$
 This implies, choosing a $\phi$ that vanishes on $[x_1,1]$: 
 $$\left\lbrace\begin{array}{l}
 \xi^{\prime\prime}=-f \quad \mbox{in } (0,x_1)\\
  \xi^\prime(0)=t_0 \\
   \xi\prime(x_1)=0 
 \end{array} \right.$$
We already know that $f(0)=-\mu u^2(0) \leq 0$, thus if $f$ does not change sign on $(0,x_1)$, we would have $\xi$ convex,
but this is incompatible with the two boundary conditions since  $\xi^\prime(0)=t_0\geq 0$. Therefore, $f$ must vanish on this interval $(0,x_1)$. The same holds true on $(x_M,1)$. Now if we count the number of zeros of $f$, seen from the side of these intervals, we have
{\it at least} $1+1$ (for each boundary interval $(0,x_1)$ and $(x_M,1)$) and $2(M-1)$ for all internal intervals, that makes
at least $2M$ zeros. But, we know from Proposition \ref{propkey} that the number of zeros of $f$ is exactly $2k$, therefore $M\leq k$
and the number of segments in the graph of $h^*$ is at most $k+1$.
\end{proof} 
In order to prove Theorem \ref{tEFHO} we need to prove two lemmas about the behavior of the Sturm-Liouville eigenfunctions on 
the vertexes of the optimal function $h^*_k$.
\begin{lemma}\label{lOC}
Let $h_k^*$ be a maximizer for the problem $\max \{\mu_k(h), h\in \mathcal{L}\}$, let $u_k$ be the eigenfunction associated to $\mu_k(h_k^*)$, and let $\{x_0,x_1,\ldots,x_m,x_{m+1}\}=suppt(h_k^{*\prime\prime})$ then:
\begin{enumerate}
\item $u_k(x_i)u^\prime(x_i)=0$ for all $i=0,...,m+1$,
\item if $h^{*\prime}_k>0$ in $(x_i,x_{i+1})$ then $u^\prime(x_i)=0$ and if  $h^{*\prime}_k<0$ in $(x_i,x_{i+1})$ then $u^\prime(x_{i+1})=0$
\end{enumerate}
\end{lemma}
\begin{proof}
In the proof, the optimal function $h_k^*$ will be simply denoted by $h$, the eigenvalue is $\mu$ and the eigenfunction by $u$. From the previous proof, we already know that $u^\prime(x_0)=u^\prime(0)=0$ and the same at the point $x_{m+1}=1$.  Now from the optimality conditions \eqref{eOCf}, and recalling that $h$ must be a piecewise linear function, we obtain that 
\begin{equation*}
\int_{x_i}^{x_{i+1}} ({u^\prime}^2 - \mu_k u^2) h dx=-h(x_{i+1})u(x_{i+1})u^\prime(x_{i+1})+h(x_i)u(x_i)u^\prime(x_i)=0.
\end{equation*}
Now we conclude by induction, indeed from the relations above, starting from $x_0=0$ and $x_1$ we obtain that $u(x_2)u'(x_2)=0$, because $h(x_2)>0$ (we know that $x_2<x_{m+1}=1$). Repeating this argument for all the intervals we conclude that $u(x_i)u'(x_i)=0$ for all $i<m$, this concludes the first part of the proof.

We now prove the second part of the Lemma. We already know that $u'(0)=0$ and $u'(1)=0$. Now we consider the interval $[x_i,x_{i+1}]$ 
with $0<x_i<x_{i+1}<1$ and we assume that $h^\prime>0$ on this interval, so we assume that we are in the increasing part of the function 
$h$. We know that $u(x_i)u^\prime(x_i)=0$, we prove that $u(x_i)\neq 0$. We integrate by parts the optimality condition 
$\int_{x_i}^{x_{i+1}} ({u^\prime}^2 - \mu_k u^2)dx=0$ and we use the relation $u^{\prime\prime}=-\frac{h^\prime}{h} u^\prime -\mu u$ (we recall that $h>0$) in order to obtain the following equality 
\begin{equation}\label{eINTCOND1}
\int_{x_i}^{x_{i+1}}u^\prime\frac{u}{h}dx=0.
\end{equation}
We integrate by parts the relation above and we obtain 
\begin{equation*}
\frac{u^2(x_{i+1})}{h(x_{i+1})}-\frac{u^2(x_i)}{h(x_i)}-\int_{x_i}^{x_{i+1}}u\frac{u'h-uh'}{h^2}dx=0,
\end{equation*}
using \eqref{eINTCOND1} we obtain that 
\begin{equation}\label{eINTCOND2}
\int_{x_i}^{x_{i+1}}u^2\frac{h'}{h^2}dx=\frac{u^2(x_i)}{h(x_i)}-\frac{u^2(x_{i+1})}{h(x_{i+1})},
\end{equation}
we know that $h'>0$ in $[x_i,x_{i+1}]$ so we conclude that $u(x_i)$ must be different from zero. We can adapt the same proof in the decreasing part of $h$, in this case we will have that $h'<0$ and we conclude that $u(x_{i+1})$ must be different from zero.
\end{proof}

\begin{lemma}\label{lOC2}
Let $h_k^*$ be a maximizer for the problem $\max \{\mu_k(h), h\in \mathcal{L}\}$, let $u_k$ be the eigenfunction associated to $\mu_k(h_k^*)$, and let $\{x_0,x_1,\ldots,x_m,x_{m+1}\}=suppt(h_k^{*\prime\prime})$ then $u_k(x_i)=0$ for all $i=1,...,m$. 
\end{lemma}
\begin{proof}
In the proof, the optimal function $h_k^*$ will be simply denoted by $h$, the eigenvalue is $\mu$ and the eigenfunction by $u$. Suppose by contradiction that $u(x_i)\neq 0$, from the point $1$ in Lemma \ref{lOC} we know that $u'(x_i)=0$. This means that there exists an interval $[a_i,b_i]$ s. t. ${u^\prime}^2-\mu u(x)<0$ inside the interval. 

 We choose a function $\phi\leq 0$ s. t. $Suppt(\phi)\subset [a_i,b_i]$ and $h+t\phi \in \mathcal{L}$, from Lemma \ref{dermuk} we obtain that 
\begin{equation*}
\dot{\mu}_{\phi}=\int_0^1 ({u'}^2 - \mu_k u^2) \phi dx=\int_{a_i}^{b_i} ({u'}^2 - \mu_k u^2) \phi dx>0
\end{equation*}
that is a contradiction with the assumption that $h$ is a maximizer.
\end{proof}
We are now ready to prove Theorem \ref{tEFHO}
\begin{proof}[Proof of Theorem \ref{tEFHO}]
In the proof, the optimal function $h_k^*$ will be simply denoted by $h$, the eigenvalue is $\mu$ and the eigenfunction by $u$. The existence follow directly from Proposition \ref{propcoco}. We start by analyzing the case $k=1$, in this case, from Theorem \ref{theok+1}, we know that $h$ can be or a constant function or a function that has a graph given by $2$ segments. From Lemma \ref{lhvanish} we now that $h$ cannot be constant, otherwise it must be constantly equal to zero. We have that $h$ must be of the following form   
\begin{equation*}
h=\begin{cases}
\frac{x}{x_a} \quad &x\in [0,x_a], \\
\frac{1-x}{1-x_a} \quad &x\in [x_a,1].
\end{cases}
\end{equation*}
We introduce the following parameter $w$ that depends on $x_a$ such that $w^2=\mu$. The explicit expression of the eigenfunction $u$ in the interval $[0,x_a]$ and in the interval $[x_a,1]$ is known (see \cite{HM21}). In particular there exist two constants $A_1$ and $B_1$ such that the eigenfunction $u$ must have the following form 
\begin{equation}\label{eEIGEN1}
u=\begin{cases}
A_1J_0(w(x_a)x) \quad &x\in [0,x_a] \\
B_1J_0(w(x_a)(1-x)) \quad &x\in [x_a,1],
\end{cases}
\end{equation}
where $J_0$ is the Bessel function with parameter $0$. From Lemma \ref{lOC2} we know that $w$ and $x_a$ must be linked by the following conditions
\begin{align*}
wx_a&=j_{0,m_1}\quad  m_1\in \mathbb{N}\\
w(1-x_a)&=j_{0,m_2}\quad  m_2\in \mathbb{N}\\
w&=j_{0,m_1}+j_{0,m_2},
\end{align*}
where $j_{0,m}$ is the $m$-th zero of the Bessel function $J_0$. We have that $u(x_a)=0$ and with the expression of $u$
given in \eqref{eEIGEN1} we see that $u$ vanishes at least $m_1-1$ times in $(0,x_a)$ and at least $m_2-1$ times in $(x_a,1)$.
The only possibility, in order not to contradict Courant nodal domain Theorem  is  $m_1=1$ and $m_2=1$. We obtain the following
\begin{equation*}
w=2j_{0,1},
\end{equation*}
and this gives that $x_a=\frac{1}{2}$, this concludes the case $k=1$.

We study the case where $k\geq 2$. From Theorem \ref{theok+1} we know that $h$  has a graph that is a polygonal line composed of (at most) $k+1$ segments. From the two lemmas \ref{lOC} and \ref{lOC2} we can conclude that $h$ must have the graph that is a polygonal line 
composed of (at most) $3$ segments. Indeed if we have more than $3$ segments this will imply that there exists a segment that is not the plateau  and it is neither the first segment nor the last. Combining the second point of Lemma \ref{lOC} and Lemma \ref{lOC2} we 
conclude that there exists a point $x_i$ s. t. $u(x_i)=u'(x_i)=0$ and this is a contradiction with the fact that $u$ is a Sturm-Liouville eigenfunction. 

This means that the graph of $h$ is composed by $2$ segments or $2$ segments plus a central plateau, we start by analyzing this second case. Let $x_a<x_b$ we know that $h$ must be of the following form    
\begin{equation*}
h=\begin{cases}
\frac{x}{x_a} \quad &x\in [0,x_a] \\
1 \quad &x\in [x_a,x_b]\\
\frac{(1-x)}{1-x_b} \quad &x\in [x_b,1].
\end{cases}
\end{equation*}
In order to simplify the notation we will define $a=x_a$ and $b=1-x_b$.

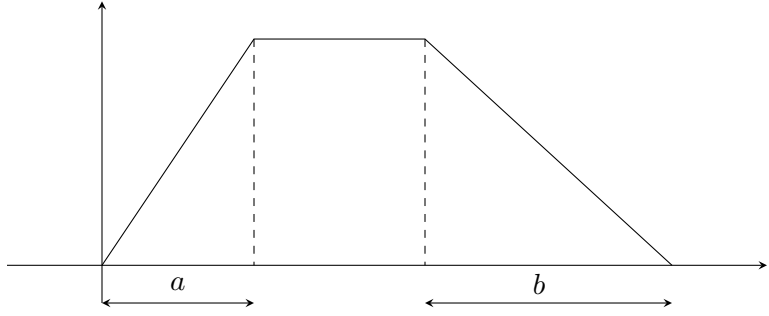
\begin{figure}[H]
\centering
\tdplotsetmaincoords{0}{0}
\begin{tikzpicture}[tdplot_main_coords, scale=2.5]
\draw (0,0,0)--(0.8,1.2,0)--(1.7,1.2,0)--(3,0,0);
\draw [-stealth](-0.5,0,0)--(3.5,0,0);
\draw [-stealth](0,-0.2,0)--(0,1.4,0);
\draw [dashed] (0.8,1.2,0)--(0.8,0,0);
\draw [dashed] (1.7,1.2,0)--(1.7,0,0);
\draw [stealth-stealth](0,-0.2,0)--(0.8,-0.2,0);
\draw [stealth-stealth](1.7,-0.2,0)--(3,-0.2,0);
\draw (0.4,-0.1,0) node {$a$};
\draw (2.3,-0.1,0) node {$b$};
\end{tikzpicture}
\caption{Shape of the function $h$, with respect to the quantities $a$ and $b$.}
\end{figure}
We introduce the following parameter $w=w(a,b)$ that depends on $a$ and $b$ and such that $w^2=\mu$. The idea of the proof is to give an implicit formula for the parameter $w$ and then, using the optimality conditions combined with an analysis of the nodal domains of the eigenfunctions $u$, we will find that only one triplet of values for $w$, $a$ and $b$ is possible.\\
The explicit expression of the eigenfunction $u$ in each interval $[0,a]$, $[a,b]$ and $[b,1]$ is known (we recover the Bessel equation on $(0,a)$ and $(b,1)$
see \cite{HM21}). Let $J_{\alpha}$ be the Bessel function with parameter $\alpha$, we know that there exist four constants $A_1$, $B_1$, $B_2$ and $C_1$ such that the eigenfunction $u$ has the following form 
\begin{equation}\label{eEIGEN2}
u=\begin{cases}
A_1J_0(wx) \quad &x\in [0,x_a] \\
B_1\cos(wx)+B_2\sin(wx) \quad &x\in [x_a,x_b]\\
C_1J_0(w(1-x)) \quad &x\in [x_b,1].
\end{cases}
\end{equation}
In order to find an implicit formula for $w$ we proceed in a classical way, we know  that the eigenfunction is $C^1$, we write the compatibility condition
and this gives a homogeneous linear system that the constants $A_1$, $B_1$, $B_2$ and $C_1$ must satisfy. The parameters $w$ are exactly the parameters for which the homogeneous linear system has a solution, so are the parameters for which the determinant of the homogeneous linear system is zero. After a straightforward computation (recalling that $J_0'=-J_1$) we find the following transcendental equation that $w$ must solve
\begin{align}\label{eTEW}
\tan(w(1-a-b)))=\frac{J_1(wa)J_0(wb)+J_1(wb)J_0(wa)}{J_1(wa)J_1(wb)-J_0(wb)J_0(wa)}.
\end{align} 
From Lemma \ref{lOC2} we conclude that 
\begin{align*}
wa&=j_{0,m_1}\quad  m_1\in \mathbb{N}\\
wb&=j_{0,m_2}\quad  m_2\in \mathbb{N},
\end{align*}
where $j_{0,m}$ is the $m$-th zero of the Bessel function $J_0$. In this case we have $m_1$ nodal interval in the interval $[0,x_a]$ and $m_2$ nodal intervals in the interval $[x_b,1]$. From the transcendental equation \eqref{eTEW} we obtain an explicit condition that $w$ must satisfies, indeed we have that $\tan(w-j_{0,m_1}-j_{0,m_2})=0$, so we obtain
\begin{equation*}
w=j_{0,m_1}+j_{0,m_2}+l\pi \quad  l\in \mathbb{Z}.
\end{equation*}
If $l<0$ then $a+b>1$ and this is not possible. If $l=0$ we have that $a+b=1$, but this is a contradiction with the fact that $x_a<x_b$. We finally obtain that 
\begin{equation*}
w=j_{0,m_1}+j_{0,m_2}+l\pi \quad  l\in \mathbb{N}\setminus \{0\}.
\end{equation*}
We analyze the behavior of the eigenfunction in the interval $[x_a,x_b]$, using the fact that $u(x_a)=u(x_b)=0$ and the formula \eqref{eEIGEN2} we obtain that 
\begin{equation*}
u(x)=\frac{B_2}{\cos(j_{0,m_1})}\big ( \sin(wx-j_{0,m_1}) \big ).
\end{equation*}
It is clear that $wx_a-j_{0,m_1}=0$ and $wx_b-j_{0,m_1}=l\pi$, this means that the eigenfunction $u$ has $l$ nodal domains in the interval $(x_a,x_b)$. We finally have that the eigenfunction $u$ has $m_1$ nodal interval in the interval $[0,x_a]$, $m_2$ nodal intervals in the interval $[x_b,1]$ and $l$ nodal domains in the interval $(x_a,x_b)$, from Courant Theorem on nodal intervals we have that $u$ has exactly $k+1$ nodal interval (Sturm-Liouville eigenfunctions are Courant sharp) so we conclude that 
\begin{equation*}
m_1+m_2+l=k+1.
\end{equation*}
From the relation above we obtain the following expression for the parameter $w$,
\begin{equation*}
w=j_{0,m_1}+j_{0,m_2}+(k+1-m_1-m_2)\pi.
\end{equation*}
We know that two consecutive zeros of the Bessel function $J_0$ always have a difference less then $\pi$ (see Lemma \ref{lzerobessel}), we conclude that the optimal $m_1$ and $m_2$ in order to maximize $w$ are $m_1=m_2=1$ and we obtain the following competitor  
\begin{equation*}
w_1=2j_{0,1}+(k-1)\pi.
\end{equation*}
We now study the case when $x_a=x_b$, in this case, using the same argument we used in order to prove the case $k=1$, we conclude that the parameter $w$ is given by the following formula 
\begin{equation*}
w=j_{0,m_1}+j_{0,m_2},\text{ with } m_1+m_2=k+1,
\end{equation*}
From the fact that two consecutive zeros of Bessel functions have a difference less then $\pi$ (see Lemma \ref{lzerobessel}), we conclude that for all $m_1$ and $m_2$ s. t. $m_1+m_2=k+1$ the following inequality holds 
\begin{equation*}
w=j_{0,m_1}+j_{0,m_2}\leq 2j_{0,1}+(k-1)\pi.
\end{equation*}
this concludes the proof.
\end{proof}

\section{Maximization of $\mu_k(h^{\alpha})$}\label{section4}
In this section we study the following maximization problem:
\begin{equation*}
\sup \{\mu_k(h^{\alpha}), h\in \mathcal{L}\}, 
\end{equation*}
with $\alpha\geq1$. In Section \ref{section3} we solved the maximization problem $\sup \{\mu_k(g), g\in \mathcal{L}\}$, that means 
with the concavity constraint on the function $g$.

In the case $\alpha=1$ we deeply use the fact that $\forall h\in \mathcal{L}$  the Sturm-Liouville problem $\mu_k(h)$ is a well defined eigenvalue problem (see Remark \ref{rmkwelldefined}), thanks to the existence of the eigenfunctions we were able to write the optimality conditions and to conclude the proof.

The big difference between the case $\alpha=1$ and the case $\alpha>1$ is the fact that, if $\alpha>1$ it is not clear if there exists an eigenfunction for the relaxed Sturm-Liouville eigenvalue $\mu_k(h^{\alpha})$. Indeed the Green Kernel associated to the Sturm-Liouville equation under consideration is the following:
\begin{equation*}
g(h;x,y)=\big (\int^{\min(x,y)}_0\frac{t}{h(t)}dt+\int_{\max(x,y)}^1\frac{1-t}{h(t)}dt\big )h(y).
\end{equation*}
We used the property that $g(h;x,y)\in L^2([0,1]\times [0,1])$ if and only if there exist $K>0$ and $p<2$ such that $h(x)\geq K(x(1-x))^p$ a. e. in $(0,1)$. If $h\in \mathcal{L}$ it is true that $g(h;x,y)\in L^2([0,1]\times [0,1])$, but if $\alpha>1$ then $g(h^{\alpha};x,y)\notin L^2([0,1]\times [0,1])$. 
This means that, in order to solve the problem $\sup \{\mu_k(h^{\alpha}), h\in \mathcal{L}\}$, we cannot directly proceed as we did for the case $\alpha=1$ . More precisely,  we have deeply used the fact that the eigenfunctions exist and also some continuity properties of $\mu_k(h)$ (see \cite{HM21}).

Let us explain the two main steps useful to solve the problem $\sup \{\mu_k(h^{\alpha}), h\in \mathcal{L}\}$:
\begin{enumerate}
\item In the first step, thanks to the introduction of an auxiliary problem, we prove that there exists a maximizer $\overline{h}$ and this maximizer has a graph that is a polygonal line composed with at most $(k+1)-$segments. Thanks to this we prove that there exists an eigenfunction for the eigenvalues $\mu_k(\overline{h}^{\alpha})$ and we give a first characterization of the eigenfunction.
\item We write the optimality conditions for the maximizer $\mu_k(\overline{h}^{\alpha})$, and we deeply use the fact that, every variation $\overline{h}+t\phi$ that we consider, if $t$ is small enough, will be a concave piecewise affine function. In particular $\mu_k((\overline{h}+t\phi)^{\alpha})$ admits an eigenfunction for $t$ small enough.  
\end{enumerate} 
The main result that we prove in this section is the following theorem: 
\begin{theorem}\label{thmax}
Let $\alpha\geq 1$ then there exists a solution $\overline{h}_{\alpha,k}$ of the following problem 
\begin{equation*}
\sup \{\mu_k(h^{\alpha}), h\in \mathcal{L}\},
\end{equation*}
and moreover

\begin{outline}
\1 if $\alpha<2$ then:
\2 $\max \{\mu_1(h^\alpha), h\in \mathcal{L}\}=\mu_1(\overline{h}_{\alpha,1}^\alpha)=(2j_{\frac{\alpha -1}{2}})^2$ and 
\begin{equation*}
\overline{h}_{\alpha,1}=\begin{cases}
2x \quad &x\in [0,\frac{1}{2}], \\
2(1-x) \quad &x\in [\frac{1}{2},1]
\end{cases}
\end{equation*}
\2 let $k\geq 2$ then $\max \{\mu_k(h), h\in \mathcal{L}\}=\mu_k(\overline{h}_{\alpha,k}^\alpha)=(2j_{\frac{\alpha -1}{2}}+(k-1)\pi)^2$ 
\begin{equation*}
\overline{h}_{\alpha,k}=\begin{cases}
\frac{x(2j_{\frac{\alpha -1}{2}}+(k-1)\pi)}{j_{\frac{\alpha -1}{2}}} \quad &x\in [0,\frac{j_{\frac{\alpha -1}{2}}}{(2j_{\frac{\alpha -1}{2}}+(k-1)\pi)}] \\
1 \quad &x\in [\frac{j_{\frac{\alpha -1}{2}}}{(2j_{\frac{\alpha -1}{2}}+(k-1)\pi)},1-\frac{j_{\frac{\alpha -1}{2}}}{(2j_{\frac{\alpha -1}{2}}+(k-1)\pi)}]\\
\frac{(1-x)(2j_{\frac{\alpha -1}{2}}+(k-1)\pi)}{j_{\frac{\alpha -1}{2}}} \quad &x\in [1-\frac{j_{\frac{\alpha -1}{2}}}{(2j_{\frac{\alpha -1}{2}}+(k-1)\pi)},1]
\end{cases},
\end{equation*}
\1 if $\alpha=2$ then: $\max \{\mu_1(h^2), h\in \mathcal{L}\}=\mu_1(\overline{h}_{2,1}^2)=((k+1)\pi)^2$. The maximizer $\overline{h}_{2,1}$ is given by any function such that $\overline{h}_{2,1}(0)=\overline{h}_{2,1}(1)=0$ and the graph of $\overline{h}_{2,1}$ is given by at most $k+1$ segments.   
\1 If $\alpha>2$ then:
\2 If $k$ is odd then $\mu_k(\overline{h}_{\alpha,k}^{\alpha})=4j_{\frac{\alpha-1}{2},\frac{k+1}{2}}^2$ and 
\begin{equation*}
\overline{h}_{\alpha,k}=\begin{cases}
2x \quad &x\in [0,\frac{1}{2}] \\
2(1-x) \quad &x\in [\frac{1}{2},1]
\end{cases}
\end{equation*}
\2 If $k$ is even then $\mu_k(\overline{h}_{\alpha,k}^{\alpha})=(j_{\frac{\alpha-1}{2},\frac{k}{2}}+j_{\frac{\alpha-1}{2},\frac{k+2}{2}})^2$ 
\begin{equation*}
\overline{h}_{\alpha,k}=\begin{cases}
\frac{(j_{\frac{\alpha-1}{2},\frac{k}{2}}+j_{\frac{\alpha-1}{2},\frac{k+2}{2}})}{j_{\frac{\alpha-1}{2},\frac{k}{2}}}x \quad &x\in [0,\frac{j_{\frac{\alpha-1}{2},\frac{k}{2}}}{(j_{\frac{\alpha-1}{2},\frac{k}{2}}+j_{\frac{\alpha-1}{2},\frac{k+2}{2}})}] \\

\frac{(j_{\frac{\alpha-1}{2},\frac{k}{2}}+j_{\frac{\alpha-1}{2},\frac{k+2}{2}})}{j_{\frac{\alpha-1}{2},\frac{k}{2}}}(1-x) \quad &x\in [\frac{j_{\frac{\alpha-1}{2},\frac{k}{2}}}{(j_{\frac{\alpha-1}{2},\frac{k}{2}}+j_{\frac{\alpha-1}{2},\frac{k+2}{2}})},1]
\end{cases}
\end{equation*}
\end{outline}
\end{theorem}
The proof of this result proceed in several steps. We start by stating the results that are useful in order to prove the existence. We first recall the following Lemma, that is a particular case of Lemma $4$ in \cite{BMO22}
\begin{lemma}[Bucur-Martinet-Oudet \cite{BMO22}]\label{lDEE}
Assume that $h,h_n\in L^{\infty}(0,1)$  satisfies $h_n \overset{\ast}{\rightharpoonup} h$ in $L^{\infty}(0,1)$ then for all $k\geq 1$
\begin{equation*}
\limsup_{n\to \infty} \mu_k(h_n)\leq \mu_k(h).
\end{equation*}
\end{lemma}
We recall that in Section \ref{section3} we use a continuity result for $\mu_k(h)$ when $h_n\rightarrow h$ in $L^2(0,1)$ proved in \cite{HM21}. We notice that this result strongly use the fact that $g(h;x,y)\in L^2([0,1]\times [0,1])$, this means that we cannot use this result for 
$\alpha >1 $. We prove a continuity result for the relaxed eigenvalues $\mu_k(h)$.
\begin{lemma}\label{lemcsl}
Let $h\in L^{\infty}(0,1)$ be a non negative function such that $\mu_k(h)<\infty$ and let $\epsilon_n$ be a sequence such that $\epsilon_n\rightarrow 0$ then 
\begin{equation*}
\lim_{n\to \infty} \mu_k(h+\epsilon_n)=\mu_k(h)
\end{equation*}
\end{lemma}  
Despite the simplicity of the statement the proof of this lemma is a little bit technical. Indeed we have that the eigenfunctions $u_{i,\epsilon}$ corresponding to $\mu_k(h+\epsilon_n)$ may present a blow up phenomenon that must be taken under control. 
\begin{proof}[Proof of Lemma \ref{lemcsl}]
We simply call the sequence $\epsilon_n$ by $\epsilon$. 
According to Lemma \ref{lDEE}, we already know that the sequence $\mu_k(h+\epsilon)$ is bounded.
Let $u_{i,\epsilon}$ be the eigenfunction associated to $\mu_i(h+\epsilon)$ with $i\leq k$, normalized by
\begin{equation*}
\int_0^1u_{i,\epsilon}^2(h+\epsilon)dx=1.
\end{equation*}
Note that we also have
\begin{align*}
\int_0^1u_{j,\epsilon}'^2(h+\epsilon)dx&=\mu_k(h+\epsilon)\\
\int_0^1u_{j,\epsilon}(h+\epsilon)dx&=0.\\
\end{align*}
So we have that the sequence $\sqrt{\epsilon}u_{j,\epsilon}$ is a bounded sequence in $H^1(0,1)$ thus 
there exists $U_j\in H^1(0,1)$ such that
\begin{align*}
\sqrt{\epsilon}u_{j,\epsilon}&\rightharpoonup U_j\quad \text{in}\quad H^1(0,1)\\
\sqrt{\epsilon}u_{j,\epsilon}&\rightarrow U_j\quad, \text{in}\quad L^2(0,1) .\\
\end{align*}
In order to proceed in the proof we need to distinguish between two different class of eigenfunctions:
\begin{itemize}
\item We call $K_1$ the set of indices $j$ for which 
\begin{equation*}
\epsilon \int_0^1u_{j,\epsilon}^2dx\rightarrow c>0.
\end{equation*}
In that case, $\int_0^1U^2_j=c>0$ and $\int_0^1U_jh=0$ so $U_j$ is not a constant function. In particular we have the 
following inequalities 
\begin{align*}
\int_0^1(\sqrt{\epsilon}u_{j,\epsilon}')^2(h+\epsilon)dx&\geq \int_0^1U'^2_jhdx+O(\epsilon)\\
\int_0^1(\sqrt{\epsilon}u_{j,\epsilon})^2(h+\epsilon)dx&= \int_0^1U^2_jhdx+O(\epsilon),\\
\end{align*}
where the first inequality follow from the lower semicontinuity of the functional with respect the weak convergence and the second equality is due to the strong $L^2$ convergence. We define the new function $v_{j,\epsilon}=\sqrt{\epsilon}u_{j,\epsilon}$, that 
is still an eigenfunction for the eigenvalue $\mu_k(h+\epsilon)$.

\item We call $K_2=\{j+1,\ldots, k$ the set of indices $l$ for which:
\begin{equation*}
\epsilon \int_0^1u_{l,\epsilon}^2dx\rightarrow 0.
\end{equation*}
In that case $U_l=0$ but for $\epsilon$ small enough, we have the following inequalities
\begin{align*}
\int_0^1(u_{l,\epsilon}')^2(h+\epsilon)dx&\geq \int_0^1(u_{l,\epsilon}')^2dx\\
\int_0^1(u_{l,\epsilon})^2(h+\epsilon)dx&= \int_0^1(u_{l,\epsilon})^2hdx+O(\epsilon).\\
\end{align*}
\end{itemize}
We now prove that for $\epsilon$ small enough the following space $Span[U_0,...,U_j,u_{j+1,\epsilon},...,u_{k,\epsilon}]$ is of dimension $k+1$. The set of functions $U_i$ with $i\in K_1$ are orthogonal, it remains to check that for $\epsilon$ small enough $U_i\notin Span[u_{j+1,\epsilon},...,u_{k,\epsilon}]$. Suppose this is not true so there exists $\overline{\epsilon}$ such that $U_i\in Span[u_{j+1,\epsilon},...,u_{k,\epsilon}]$ for all $\epsilon<\overline{\epsilon}$. In particular there exists $\alpha_{j+1,\epsilon},..., \alpha_{k,\epsilon}$ such that 
\begin{equation*}
U_i=\sum_{m=j+1}^k \alpha_{m,\epsilon}u_{m,\epsilon}
\end{equation*}

and there exists an index $l$ such that $\alpha_{l,\epsilon}\rightarrow c\neq 0$. So 
\begin{equation}\label{eqcontrcont}
\lim_{\epsilon\to 0}\int_0^1U_iu_{l,\epsilon}(h+\epsilon)dx\rightarrow c\neq 0.
\end{equation} 
Now we recall that $\int_0^1v_{i,\epsilon}u_{l,\epsilon}(h+\epsilon)=0$ for all $\epsilon$ so in particular:
\begin{align*}
|\int_0^1U_iu_{l,\epsilon}(h+\epsilon)dx|&\leq\int_0^1|U_iu_{l,\epsilon}(h+\epsilon)-v_{i,\epsilon}u_{l,\epsilon}(h+\epsilon)|dx\\ 
&\leq ||v_{i,\epsilon}-U_i||_{L^2}^{\frac{1}{2}}|u_{l,\epsilon}(h+\epsilon)||_{L^2}^{\frac{1}{2}}\rightarrow 0,
\end{align*}
That is a contradiction with \eqref{eqcontrcont}, so we proved that the space $Span[U_0,...,U_j,u_{j+1\epsilon},...,u_{k,\epsilon}]$ is of dimension $k+1$.

We now conclude the proof, indeed from the inequalities above, for $\epsilon$ small enough we have that 
\begin{align*}
\mu_k(h+\epsilon)&=\max_{\beta\in \mathbb{R}^{k+1}}\frac{\sum_{i\in K_1}\beta_i^2\int_0^1v_{i,\epsilon}'^2(h+\epsilon)dx+\sum_{l\in K_2}\beta_l^2\int_0^1u_{l,\epsilon}'^2(h+\epsilon)dx}{\sum_{i\in K_1}\beta_i^2\int_0^1v_{i,\epsilon}^2(h+\epsilon)dx+\sum_{l\in K_2}\beta_l^2\int_0^1u_{l,\epsilon}^2(h+\epsilon)dx}\\
&\geq \sup_{w\in Span[U_0,...,U_j,u_{j+1\epsilon},...,u_{k,\epsilon}]} \frac{\int_0^1w'^2hdx}{\int_0^1w^2hdx}+O(\epsilon)\\
&\geq \mu_k(h)+O(\epsilon).
\end{align*}
We finally proved that $\liminf \mu_k(h+\epsilon)\geq \mu_k(h)$, the uppersemicontinuity follow from Lemma \ref{lDEE}.
\end{proof}
We now prove the existence of a maximizer and we give a first characterization for the maximizer $\overline{h}_k$.
\begin{lemma}\label{lfirstcar}
The following maximization problem 
\begin{equation*}
\sup \{\mu_k(h^{\alpha}), h\in \mathcal{L}\},
\end{equation*}
has a solution $\overline{h}_{\alpha,k}$ that is a concave piece-wise affine function with at most $(k+1)$ affine pieces such that $\overline{h}_{\alpha,k}(0)=\overline{h}_{\alpha,k}(1)=0$. and defined as in \eqref{defhaff} below. Moreover the function
\begin{equation*}
\overline{u}_{\alpha,k}=
\begin{cases}
x^{\frac{1-\alpha}{2}}A_1J_{\frac{\alpha-1}{2}}(wx)\quad x\in [0,x_1] \\
\vdots\\
(x+\frac{b_{i-1}}{a_{i-1}})^{\frac{1-\alpha}{2}}\big [A_iJ_{\frac{\alpha-1}{2}}\big (w(x+\frac{b_{i-1}}{a_{i-1}})\big)+B_{i-1}Y_{\frac{\alpha-1}{2}}\big (w(x+\frac{b_{i-1}}{a_{i-1}})\big)\big ]\quad x\in [x_{i-1},x_i]\\
\vdots\\
A_j\cos(wx)+B_{j-1}\sin(wx)\quad x\in [x_{j-1},x_j]\\
\vdots\\
(1-x+\frac{b_{l-1}}{a_{l-1}})^{\frac{1-\alpha}{2}}\big [A_lJ_{\frac{\alpha-1}{2}}\big (w(1-x+\frac{b_{l-1}}{a_{l-1}})\big)+B_{l-1}Y_{\frac{\alpha-1}{2}}\big (w(1-x+\frac{b_{l-1}}{a_{l-1}})\big)\big ]\
\qquad x\in [x_{l-1},x_l] \\
\vdots\\
(1-x)^{\frac{1-\alpha}{2}}A_{k+1}J_{\frac{\alpha-1}{2}}(w(1-x))\quad x\in [x_k,1],
\end{cases}
\end{equation*} 
is the eigenfunction associated to the eigenvalue $\mu_k(\overline{h}_{\alpha,k}^{\alpha})$. Here $J_\alpha$ and $Y_\alpha$ are the Bessel function of the first and second kind with index $\alpha$, $w^2=\mu_k(\overline{h}_{\alpha,k}^{\alpha})$ and $\{A_i\}_{i=1}^{k+1}$, $\{B_i,\}_{i=1}^{k-1}$ are constants. 
\end{lemma}
In order to prove this result, as we already pointed out, we cannot directly use the same techniques as in Section \ref{section3}, 
because in general we don't have existence of eigenfunctions. For this reason we proceed via an approximation argument. We construct a maximizing sequence and we identify what is the limit of this maximizing sequence, proving in this way the existence and the 
qualitative property of the maximizer. The maximizing sequence is constructed by solving an auxiliary maximization problem for which we can use the techniques developed in Section \ref{section3}. 
\begin{proof}[Proof of Lemma \ref{lfirstcar}]
In order to simplify the notation we will not write the subscripts $\alpha$ and $k$. 
Let $\epsilon>0$ we introduce the following space of functions 
\begin{equation*}
\mathcal{L}_{\epsilon}=\{h\in \mathcal{L} \,,\,\,h\geq \epsilon\}.
\end{equation*}
We introduce an auxiliary maximization problem that is $\sup \{\mu(h^{\alpha}), h\in \mathcal{L}_{\epsilon}\}$. Now for any 
$h\in\mathcal{L}_{\epsilon}$ there exists an eigenfunction associated to the eigenvalue $\mu(h^{\alpha})$, because $h\geq \epsilon$.
For this reason we can now adapt the results we present in Section \ref{section3}. The existence of a maximizer in $\mathcal{L}_{\epsilon}$, that we denote by $\overline{h}_{\epsilon}$, follow by an easy adaptation of Proposition \ref{propcoco}. It is straightforward to check that the derivative of $\mu(h^{\alpha})$ in the direction of $\phi$ is given by 
\begin{equation}\label{dermukd}
\dot{\mu}_{\phi}(h^{\alpha})=\alpha\int_0^1 ({u'}^2 - \mu(h^{\alpha}) u^2)h^{\alpha-1}\phi dx.
\end{equation}
We have that the role of the function $f={u'}^2 - \mu u^2$ is now played by the new function $f_{\alpha}=\alpha ({u'}^2 - \mu(h^{\alpha}) u^2)h^{\alpha-1}$. We need to study the nodal interval of the function $f_{\alpha}$, but the function $\alpha h^{\alpha-1}$ is positive, ,
 so we reduce another time to the study of the nodal interval of the function ${u'}^2 - \mu(h^{\alpha}) u^2$. By a straightforward adaptation of Theorem \ref{theok+1} and Proposition \ref{lhvanish} we have that $\overline{h}_{\epsilon}$ is a concave function that has a graph that is a polygonal line of at most $(k+1)$ segments and such that $\overline{h}_{\epsilon}(0)=\overline{h}_{\epsilon}(1)=\epsilon$. In particular there exist positive parameters $\{a_{i,\epsilon}\}_{i=1}^{k+1}$, $\{b_{i,\epsilon}\}_{i=1}^{k-1}$ and $x_{i,\epsilon}\in [0,1]$ with $i=1,...,k$ such that:
\begin{equation*}
\overline{h}_{\epsilon}=\begin{cases}
a_{1,\epsilon}x+\epsilon \quad x\in [0,x_{1,\epsilon}]\\
\vdots\\
a_{i,\epsilon}x+b_{i-1,\epsilon}\quad x\in [x_{i-1,\epsilon},x_{i,\epsilon}]\\
\vdots\\
1 \quad x\in [x_{j-1,\epsilon},x_{j,\epsilon}]\\
\vdots\\
a_{l,\epsilon}(1-x)+b_{l-1,\epsilon}\quad x\in [x_{l-1,\epsilon},x_{l,\epsilon}]\\
\vdots\\
a_{k+1,\epsilon}(1-x)+\epsilon \quad x\in [x_{k,\epsilon},1].
\end{cases}
\end{equation*}
Sending $\epsilon\rightarrow 0$, for the expression above we can conclude that, up to some subsequence 
$\overline{h}_{\epsilon}\overset{\ast}{\rightharpoonup} \overline{h}$ in $L^{\infty}(0,1)$, where $\overline{h}$ is given by 
\begin{equation}\label{defhaff}
\overline{h}=\begin{cases}
a_{1}x \quad x\in [0,x_{1}]\\
\vdots\\
a_{i}x+b_{i-1}\quad x\in [x_{i-1},x_{i}]\\
\vdots\\
1 \quad x\in [x_{j-1},x_{j}]\\
\vdots\\
a_{l}(1-x)+b_{l-1}\quad x\in [x_{l-1},x_{l}]\\
\vdots\\
a_{k+1}(1-x) \quad x\in [x_{k},1].
\end{cases}
\end{equation}
Let $a\neq 0$, $w\neq 0$ and $b\in \mathbb{R}$ a direct computation shows that the solution of the following equation:
\begin{equation*}
-\frac{d}{dx}\big((ax+b)^{\alpha}\frac{dg}{dx}(x)\big)=w^2(ax+b)^{\alpha}g(x)
\end{equation*}
is given by 
\begin{equation*}
g(x)=(x+\frac{b}{a})^{\frac{1-\alpha}{2}}\big [AJ_{\frac{\alpha-1}{2}}\big (w(x+\frac{b}{a})\big)+BY_{\frac{\alpha-1}{2}}\big (w(x+\frac{b}{a})\big)\big ],
\end{equation*}
where $A$ and $B$ are two constants and  $J_\alpha$, $Y_\alpha$ are the Bessel function of the first and second kind with index $\alpha$. Thanks to this observation we can give an explicit expression for the eigenfunction $\overline{u}$, we only write the expression in the first and the last interval:
\begin{equation*}
\overline{u}=\begin{cases}
(x)^{\frac{1-\alpha}{2}}\big [A_1J_{\frac{\alpha-1}{2}}\big (wx\big)+B_1Y_{\frac{\alpha-1}{2}}\big (wx)\big)\big ] \quad x\in[0,x_1]\\
\vdots\\
(1-x)^{\frac{1-\alpha}{2}}\big [A_{k+1}J_{\frac{\alpha-1}{2}}\big (w(1-x)\big)+B_{k+1}Y_{\frac{\alpha-1}{2}}\big (w(1-x))\big)\big ] \quad x\in[x_k,1].\\
\end{cases}
\end{equation*}
We now prove that $B_1=B_{k+1}=0$, indeed if $B_1\neq 0$ then $\lim_{x\to 0}\overline{h}^{\alpha}\overline{u}\neq 0$ that is in contradiction with the boundary condition, the same argument works for $B_{k+1}$.

In order to conclude the proof we need to prove that $\overline{h}$ is the maximizer for the problem $\sup \{\mu_k(h^{\alpha}), h\in \mathcal{L}\}$. Suppose by contradiction that there exists a function $f\in \mathcal{L}$ such that $\mu_k(f^{\alpha})>\mu_k(\overline{h}^{\alpha})$. We consider the sequence of functions $f^{\alpha}+\epsilon$, then we have that $\mu_k(f^{\alpha}+\epsilon)\leq\mu_k(\overline{h}^{\alpha}_{\epsilon})$. Passing to the limit and using Lemmas \ref{lDEE} and \ref{lemcsl} we obtain that $\mu_k(f^{\alpha})\leq \mu_k(\overline{h}^{\alpha})$ that is a contradiction. 
\end{proof}
Now we want to give the exact expression of the function $\overline{h}_{\alpha,k}$ in order to prove Theorem \ref{thmax}. For that we need to distinguish the case $\alpha\neq 2$ and $\alpha=2$. We give the analogous in this setting of Lemmas \ref{lOC} and \ref{lOC2}
\begin{lemma}\label{lOCk}
Let $\overline{h}_{\alpha,k}$ be a maximizer for the problem $\sup \{\mu_k(h^{\alpha}), h\in \mathcal{L}\}$ let $u_k$ be the eigenfunction associated to $\mu_k(\overline{h}_{\alpha,k}^{\alpha})$ and let $\{x_0,x_1,\ldots,x_m,x_{m+1}\}=suppt(\overline{h}_{\alpha,k}^{\prime\prime})$ then:
\begin{enumerate}
\item $u_k(x_i)u_k^\prime(x_i)=0$ for all $i=0,...,m+1$.
\item Let $\alpha>2$, if $\overline{h}^{\prime}_{\alpha,k}>0$ in $(x_i,x_{i+1})$ then $u^\prime(x_{i+1})=0$ and if  $\overline{h}^{\prime}_{\alpha,k}<0$ in $(x_i,x_{i+1})$ then $u^\prime(x_{i})=0$ for all $i=1,...,m-1$
\item Let $\alpha<2$, if $\overline{h}^{\prime}_{\alpha,k}>0$ in $(x_i,x_{i+1})$ then $u^\prime(x_{i})=0$ and if  $\overline{h}^{\prime}_{\alpha,k}<0$ in $(x_i,x_{i+1})$ then $u^\prime(x_{i+1})=0$ for all $i=1,...,m-1$
\end{enumerate} 
\end{lemma}
\begin{proof}
The first point of the proof is a straightforward adaptation of Lemma \ref{lOC}. Indeed in this case we have that $f_{\alpha}=\alpha({u'}^2 - \mu(h^{\alpha}) u^2)h^{\alpha-1}$ thanks to the optimality conditions we will obtain, recalling that $h$ is piece-wise affine  
\begin{equation*}
0=\int_{x_i}^{x_{i+1}}f_{\alpha}h=\int_{x_i}^{x_{i+1}} ({u^\prime}^2 - \mu_k u^2) h^{\alpha} dx=-h^{\alpha}(x_{i+1})u(x_{i+1})u^\prime(x_{i+1})+h^{\alpha}(x_i)u(x_i)u^\prime(x_i)
\end{equation*}
then the conclusion follows directly from the same argument as in Lemma \ref{lOC}.

For the second part of the Lemma, the idea of the proof is the same of Lemma \ref{lOC}, but the computations are different. from the optimality conditions we know that $\int_{x_i}^{x_{i+1}}({u^\prime}^2 - \mu_k u^2) h^{\alpha-1}=0$ and $u$ solves the equation (recall that we are in the region where $h>0$) $u^{\prime\prime}=-\alpha\frac{h^\prime}{h} u^\prime -\mu u$ we obtain:
\begin{equation*}
0=\int_{x_i}^{x_{i+1}}({u^\prime}^2 - \mu_k u^2) h^{\alpha-1}=-\int_{x_i}^{x_{i+1}}u[-\alpha u'h'h^{\alpha-2}+(\alpha-1)h^{\alpha-2}u'].
\end{equation*}
recalling that $h'$ is constant on $[x_i,x_{i+1}]$ we get that $\int_{x_i}^{x_{i+1}}uu'h^{\alpha-2}=0$, integrating by parts we obtain:
\begin{equation*}
\frac{1}{2}u^2(x_{i+1})h^{\alpha-2}(x_{i+1})-\frac{1}{2}u^2(x_{i})h^{\alpha-2}(x_{i})=\frac{1}{2}(\alpha-2)\int_{x_i}^{x_{i+1}}u^2h'h^{\alpha-2}>0.
\end{equation*}
For $\alpha>2$ this implies $u(x_{i+1})\neq 0$ that, thanks to the first part if the Lemma we conclude $u'(x_{i+1})=0$. For $\alpha<2$ we conclude that $u'(x_{i})=0$. In the region where $\overline{h}^{\prime}_k<0$ we can do the same computations obtaining  $u'(x_{i})=0$ for $\alpha>2$ and $u'(x_{i+1})=0$ for $\alpha<2$
\end{proof}
\begin{lemma}\label{lOC2k}
Let $\overline{h}_{\alpha,k}$ be a maximizer for the problem $\sup \{\mu_k(h^{\alpha}), h\in \mathcal{L}\}$ let $u_k$ be the eigenfunction associated to $\mu_k(\overline{h}_k^{\alpha})$ and let $\{x_0,x_1,\ldots,x_m,x_{m+1}\}=suppt(\overline{h}_{\alpha,k}^{\prime\prime})$ then $u_k(x_i)=0$ for all $i=1,...,m$.
\end{lemma}

\begin{proof}
We start by noticing that, as a byproduct in the proof of Lemma \ref{lfirstcar}, we also show that if $h\in \mathcal{L}$ is a piece-wise affine function with a finite number of affine part, then there exists an eigenfunction associated to the eigenvalue $\mu_k(h^{\alpha})$. 
Thanks to this, and from the fact that for all perturbations $\phi$ that we take into account for all $t$ small enough the function 
$h+t\phi$ is is a piece-wise affine function with a finite number of affine part, we can easily adapt the proof of the Lemma \ref{lOC2} (we can use the same perturbation function $\phi$).  
\end{proof}
We now give a bound on the number of the segments composing the graph of $\overline{h}_{\alpha,k}$ when $\alpha\neq2$.
\begin{lemma}\label{lbound3}
Let $\alpha\neq2$ and let $\overline{h}_{\alpha,k}$ be a maximizer for the problem $\sup \{\mu_k(h^{\alpha}), h\in \mathcal{L}\}$ then $\overline{h}_{\alpha,k}$ has a graph that is a polygonal line composed of (at most) $3$ segments. 
\end{lemma}
\begin{proof}
We have  $\alpha\neq2$ then the results follow directly from the second point of Lemma \ref{lOCk} and from Lemma \ref{lOC2k}, knowing that it is not possible that there exists a point $x\in (0,1)$ for which $u'(x)=u(x)=0$. 
\end{proof}
The last ingredient  in order to prove Theorem \ref{thmax} is the following Lemma 
\begin{lemma}\label{lzerobessel}
Let $j_{\nu,n}$ be the $n-$th zero of the Bessel function $J_\nu$ if $|\nu|\geq \frac{1}{2}$ then 
\begin{align*}
j_{\nu,n+1}-j_{\nu,n}&\geq \pi\\
j_{\nu,n+1}-j_{\nu,n}&\geq j_{\nu,n+2}-j_{\nu,n+1}.
\end{align*}
the inequalities are reversed if $|\nu|< \frac{1}{2}$.
\end{lemma}
\begin{proof}
See Theorem $1.6$ in \cite{S01} and Lemma $2.9$ in \cite{LN07}
\end{proof} 
We are now ready to prove Theorem \ref{thmax}
\begin{proof}[Proof of Theorem \ref{thmax}] We start with the case $\alpha\neq2$.
Thanks to Lemma  \ref{lbound3} we only need to study the case where the maximizer $h$ is composed of $2$ or $3$ segments. The idea 
of the proof is exactly the same as the proof of Theorem \ref{tEFHO}, we will use also the same notation.
The general form of the eigenfunction $u$ is the following 
\begin{equation*}
u=\begin{cases}
x^{\frac{1-\alpha}{2}}A_1J_{\frac{\alpha-1}{2}}(wx) \quad &x\in [0,x_a] \\
B_1\cos(wx)+B_2\sin(wx) \quad &x\in [x_a,x_b]\\
(1-x)^{\frac{1-\alpha}{2}}C_1J_{\frac{\alpha-1}{2}}((w(1-x)) \quad &x\in [x_b,1].
\end{cases}
\end{equation*}
In order to find an implicit formula for $w$ we proceed in a classical way, we know  that the eigenfunction is $C^1$, we write the compatibility condition
and this gives a homogeneous linear system that the constants $A_1$, $B_1$, $B_2$ and $C_1$ must satisfy. The parameters $w$ are those 
for which the homogeneous linear system has a solution, so are the parameters for which the determinant of the homogeneous linear system is zero.  In order to simplify the notation we introduce the function 
\begin{equation*}
p(x)=\frac{d}{dx}\big (x^{\frac{1-\alpha}{2}}J_{\frac{\alpha-1}{2}}(wx) \big),
\end{equation*} 
we notice that $g(a)\neq 0$ and $g(b)\neq 0$. A straightforward computation yields
\begin{equation*}
\tan(w(1-a-b)))=\frac{-p(a)b^{\frac{1-\alpha}{2}}J_{\frac{\alpha-1}{2}}(wb)+p(b)a^{\frac{1-\alpha}{2}}J_{\frac{\alpha-1}{2}}(wa)}{-p(a)p(b)+a^{\frac{1-\alpha}{2}}J_{\frac{\alpha-1}{2}}(wa)b^{\frac{1-\alpha}{2}}J_{\frac{\alpha-1}{2}}(wb)}.
\end{equation*}
From Lemma \ref{lOC2k} we conclude that 
\begin{align}\label{ewa}
wa&=j_{\frac{\alpha-1}{2},m_1}\quad  m_1\in \mathbb{N}\\ \label{ewb}
wb&=j_{\frac{\alpha-1}{2},m_2}\quad  m_2\in \mathbb{N}, 
\end{align}
where $j_{\frac{\alpha-1}{2},m}$ is the $m$-th zero of the Bessel function $J_{\frac{\alpha-1}{2}}$. In order to analyze the plateau we proceed as in the proof of Theorem \ref{tEFHO} and we conclude that:
\begin{equation}\label{econdw}
w=j_{\frac{\alpha-1}{2},m_1}+j_{\frac{\alpha-1}{2},m_2}+(k+1-m_1-m_2)\pi.
\end{equation}
with the condition $k+1\geq m_1+m_2$. Suppose now $k=1$, therefore $m_1=m_2=1$ and in this case  conclude that:
\begin{equation*}
w=2j_{\frac{\alpha -1}{2}}.
\end{equation*}
From \eqref{econdw} and from Lemma \ref{lzerobessel} we obtain the following situation:
\begin{itemize}
\item if $\alpha<2$ and $k\geq 2$ we have that $w=2j_{\frac{\alpha -1}{2}}+(k-1)\pi$.
\item if $\alpha>2$ and $k$ is odd we have that $w=2j_{\frac{\alpha-1}{2},\frac{k+1}{2}}$.
\item if $\alpha>2$ and $k$ is even we have that $w=j_{\frac{\alpha-1}{2},\frac{k}{2}}+j_{\frac{\alpha-1}{2},\frac{k+2}{2}}$.
\end{itemize}
The shape of the maximizer now follows from the above expressions for $w$ and from \eqref{ewa} and \eqref{ewb}.  

It remains to study the case $\alpha=2$. From Lemma \ref{lfirstcar} we know that $\overline{h}_{2,k}$ has the graph composed at most of $k+1$ segments. In this case we have that a general solution for the equation 
\begin{equation*}
-\frac{d}{dx}\big((ax+b)^{2}\frac{du}{dx}(x)\big)=w^2 (ax+b)^{2}u(x)
\end{equation*}
is given by 
\begin{equation*}
u(x)=\frac{1}{x+\frac{b}{a}}\big [A\sin\big (w(x+\frac{b}{a})\big)+B\cos\big (w(x+\frac{b}{a})\big)\big ].
\end{equation*}
From the above equation is clear that the distance between two consecutive zeros of the eigenfunction $u_k$ is constantly equal to $\pi/w$.
Let $\{0,x_1,\ldots,x_m,1\}=suppt(\overline{h}_{\alpha,k}^{\prime\prime})$ then, from Lemma \ref{lOC2k} we know that $u(x_i)=0$. 
We conclude that $wx_i=k_i\pi$ for all $i=1,\ldots,m$ and $w*1=k_{m+1} \pi$., We also know that $u$ must have $k+1$ nodal intervals, this implies that $k_{m+1}=k+1$. 
We finally obtain $w=(k+1)\pi$. From the above considerations we also conclude that the maximizer $\overline{h}_{2,k}$ is given by any function that has the graph composed at most of $k+1$ segments on a subdivision
$\{0,x_1,\ldots,x_m,1\}$ with $x_i=k_i/(k+1)$ and such that $\overline{h}_{2,k}(0)=\overline{h}_{2,k}(x_i)=\overline{h}_{2,k}(1)=0$.
\end{proof}
\section{Sharp upper bounds for $D(\Omega)^2\mu_k(\Omega)$}\label{section5}
In this section we want to state Theorem \ref{tmain} in a more precise way, giving a precise description of the collapsing sequence of domains for which the inequality is saturated. The proof of this result will directly follow from Lemma \ref{lLBMU}, Proposition \ref{equalsup} and from Theorem \ref{thmax}
\begin{theorem}\label{tmain2}
Let $\Omega\subset \mathbb{R}^d$ be a domain, let $g$ be the profile function associated to $\Omega$. If the function $g$ is a optimal $\frac{1}{\alpha}$-concave function with $\alpha\geq 1$, then the following bounds hold:
\begin{outline}
\1 let $\alpha<2$ then: $D(\Omega)^2\mu_k(\Omega)\leq (2j_{\frac{\alpha -1}{2}}+(k-1)\pi)^2$ 

\1 let $\alpha=2$ then: $D(\Omega)^2\mu_k(\Omega)\leq ((k+1)\pi)^2$
\1 let $\alpha>2$ then:
\2 if $k$ is odd then $D(\Omega)^2\mu_k(\Omega)\leq 4j_{\frac{\alpha-1}{2},\frac{k+1}{2}}^2$ 
\2 if $k$ is even then $D(\Omega)^2\mu_k(\Omega)\leq (j_{\frac{\alpha-1}{2},\frac{k}{2}}+j_{\frac{\alpha-1}{2},\frac{k+2}{2}})^2$ 
\end{outline}
where $j_{\nu,m}$ is the $m-$th zero of the Bessel function $J_{\nu}$. Moreover in all the above cases the equality is achieved in the limit by a sequence of collapsing domains $\Omega_{\epsilon, \overline{h}_{\alpha,k}}$ with profile function given by $\epsilon^{d-1}\overline{h}_{\alpha,k}$, where the functions $\overline{h}_{\alpha,k}$ are defined in Theorem \ref{thmax}.
\end{theorem}
\begin{proof}
From Proposition \ref{equalsup} is it clear that:
\begin{equation*}
\sup\{D(\Omega)^2\mu_k(\Omega), \Omega\subset \mathbb{R}^d \mbox{ with profile function $h$ } \frac{1}{\alpha} \mbox{-concave} \}= \sup \{\mu_k(h^{\alpha}), h\in \mathcal{L}\}.
\end{equation*}
The proof is now a direct consequence of Theorem \ref{thmax}. Moreover from Lemma \ref{lLBMU} we conclude that the equality is achieved in the limit by a sequence of collapsing domains $\Omega_{\epsilon, \overline{h}_{\alpha,k}}$ with profile function given by $\epsilon^{d-1}\overline{h}_{\alpha,k}$, where the functions $\overline{h}_{\alpha,k}$ are defined in Theorem \ref{thmax}.
\end{proof}
As a particular case of Theorem \ref{tmain2} we obtain sharp upper bounds for $D(\Omega)^2\mu_k(\Omega)$ in the convex setting. Indeed,
 let $\Omega\subset\mathbb{R}^d$ be a convex domain, then, by Brunn-Minkowski inequality, its
  profile function $g$ is $\frac{1}{d-1}-$concave function. We obtain in this way a new proof of the Kr\"oger inequalities \cite{K99} 

\section{Examples of unboundedness of $D(\Omega)^2\mu_1(\Omega)$}\label{section6}
In this Section we want to present some examples of sequences of domains $\Omega_{\epsilon}\subset \mathbb{R}^d$ such that:
\begin{equation}
D(\Omega_{\epsilon})^2\mu_1(\Omega)\rightarrow \infty.
\end{equation}
We first notice that the class of domains we introduced in the Introduction is an optimal class of domains in order to have boundedness 
of the quantity $D(\Omega)^2\mu_1(\Omega)$. Indeed let $M$ be a given arbitrary number, now there exists a number  $\alpha$ such that $M<4j_{\frac{\alpha-1}{2},1}$. From Theorem \ref{tmain2} we conclude that there exists a domain $\Omega$ with a profile function that
is $\frac{1}{\alpha}-$concave and such that $D(\Omega_{\epsilon})^2\mu_k(\Omega)\geq M$ yielding a first example.

Another explicit example of a sequence of domains for which $D(\Omega_{\epsilon})^2\mu_k(\Omega)$ is unbounded is also given in Example 1.3.7 in \cite{HKP16}.
We give now another explicit example of a sequence of domains for which $D(\Omega_{\epsilon})^2\mu_k(\Omega)$ is unbounded.
We first prove the following equality
\begin{equation*}
\sup \{\mu_k(h), h\in L^\infty(0,1), h\geq 0,  h\not= 0 \}=+\infty.
\end{equation*}
For that purpose, we want to construct a particular sequence of functions:
let $0<a<\frac{1}{2}$ and let us denote by $w_a$ the first positive root of the following equation
\begin{equation}\label{equationwa}
\tan\big (x\big (\frac{1}{2} - a\big )\big )=1+\frac{1}{a x}.
\end{equation}
It is straightforward to obtain the following estimates for $w_a$: 
\begin{equation}\label{eULBMU}
\frac{\pi}{4(\frac{1}{2}-a)}< w_a< \frac{\pi}{3(\frac{1}{2}-a)}
\end{equation}
Now we define the following function
\begin{equation*}
g_a(x)=\begin{cases}
e^{2w_a(x-a)}\quad &0\leq x<a\\
1 \quad &a\leq x\leq 1-a \\
e^{2w_a(1-x-a)} \quad &1-a< x\leq 1.
\end{cases}
\end{equation*}
We have 
\begin{proposition}\label{lMUINF}
Let $g_a$ be the function defined above, then $\mu_1(g_a)=w_a$ and therefore
\begin{equation*}
\lim_{a\rightarrow \frac{1}{2}} \mu_1(g_a)=+\infty
\end{equation*}
\end{proposition}
\begin{proof}
The function $g_a$ is strictly positive, so we can identify the quantity $\mu_1(g_a)$ as the first non trivial eigenvalue of the following eigenvalue problem
\begin{equation}\label{eigenpb}
\begin{cases}
\vspace{0.3cm}
   -\frac{d}{dx}\big(g_a(x)\frac{du}{dx}(x)\big)=\mu(g_a) g_a(x)u(x)  \qquad  x\in \big(0,1\big) \\
     \frac{du}{dx}(0)=\frac{du}{dx}(1)=0,
\end{cases}
\end{equation}
in order to simplify the notation we will define $\mu_a=\mu_1(g_a)$. We divide the interval $(0,\frac{1}{2}]$ in two intervals $(0,a]$ and $[a,\frac{1}{2}]$. We solve the equation in $(0,\frac{1}{2}]$ and, by symmetry of $g_a$, the associated eigenfunction $u$ must be odd with respect to the point $x=\frac{1}{2}$.
\begin{enumerate}
\item In the interval $[a,\frac{1}{2}]$, knowing that $u(\frac{1}{2})=0$, we have that the eigenfunction, in this interval, must be of the form
\begin{equation*}
u_r=\alpha \sin\big (\sqrt{\mu_a}\big (x-\frac{1}{2}\big )\big ),
\end{equation*} 
where $\alpha \in \mathbb{R}$
\item In the region $(0,a]$ we have to solve the differential equation 
\begin{equation}\label{odelin}
\frac{d^2u}{dx^2}+2w_a\frac{du}{dx}+\mu_a u=0.
\end{equation}
We have two possibilities, the first is when $\mu_a\neq w_a^2$ and the second is when $\mu_a=w_a^2$. We now continue the analysis assuming that the second case 
happens, we will exclude the first possibility later in the proof. So in the case when $\mu_a=w_a^2$, knowing that $u'(0)=0$ we obtain that 
\begin{equation*}
u_l(x)=A e^{-\sqrt{\mu_a}x}(1+\sqrt{\mu_a}x),
\end{equation*}   
where $A \in \mathbb{R}$
\end{enumerate}

We impose the conditions on $\alpha$ and $A$ in order to have $u_l(a)=u_r(a)$ and $u_l'(a)=u_r'(a)$, so we need to solve the following system 
\begin{equation*}
\begin{cases}
(1+\sqrt{\mu_a}a)Ae^{-\sqrt{\mu_a}a}-\alpha\sin\big (\sqrt{\mu_a}\big (a-\frac{1}{2}\big )\big )=0 \\
-a\mu_ae^{-\sqrt{\mu_a}a}-\alpha\sqrt{\mu_a}\cos\big (\sqrt{\mu_a}\big (a-\frac{1}{2}\big )\big )=0.
\end{cases}
\end{equation*}
Necessarily the determinant of this linear system is equal to zero, this gives us a transcendental equation, the first non trivial root of this equation will be the first eigenvalue $\mu_a$. The system above has a solution if and only if 
\begin{equation*}
\tan\big (\sqrt{\mu_a}\big (a-\frac{1}{2}\big )\big )=1+\frac{1}{a\sqrt{\mu_a}}.
\end{equation*}
and we recover Equation \eqref{equationwa} confirming that, in the case, $\mu_a=w_a^2$ we have obtained an eigenfunction of Problem \eqref{eigenpb}.

It remains to prove that we cannot find a smaller eigenvalue. Let us assume, for a contradiction, that there exists an eigenvalue $\mu_a<w_a^2$.
Let $\rho_{1,2}=-w_a\pm \sqrt{w_a^2-\mu_a}$ then the solution  of the differential equation \eqref{odelin} in $(0,a]$, knowing that $u'(0)=0$, must be of the form 
\begin{equation*}
u_r(x)=A(\rho_1e^{\rho_2x}-\rho_2e^{\rho_2x}),
\end{equation*}
where $A\in \mathbb{R}$. Imposing the condition that the eigenfunction must be $\mathcal{C}^1$ we obtain the following  transcendental equation that $\mu_a$ must solve 
\begin{equation}\label{eTEMU}
\sqrt{\mu_a}\tan\big (\sqrt{\mu_a}\big (\frac{1}{2} - a\big )\big )=w_a+\frac{\sqrt{w_a^2-\mu_a}}{\tanh(a\sqrt{w_a^2-\mu_a})}. 
\end{equation} 
Let us introduce the function
\begin{equation*}
\phi(x)=x\tan\big (x\big (\frac{1}{2} - a\big )\big )-w_a-\frac{\sqrt{w_a^2-x^2}}{\tanh(a\sqrt{w_a^2-x^2})}<0 ,
\end{equation*}
It is immediate to check that this function is strictly increasing in the interval $(0,w_a)$. Moreover, by definition of $w_a$, we have
$$\lim_{x\to w_a} \phi(x) =0$$
thus, $\phi(x)<0$ for all $x\in (0,w_a)$ showing that there are no eigenvalues in the interval $(0,w_a)$.

We have proved that $\mu_a=w_a^2$ and the last claim of the proposition comes from the estimate \eqref{eULBMU}.
\end{proof}
We now construct the sequence of domains for which $D(\Omega)^2\mu_k(\Omega)$ is unbounded. Let $\Omega_{\epsilon, g_a}$ be a domain with profile function given by $\epsilon^{d-1}g_a$. We fix a constant $M>0$, from Lemma \ref{lMUINF} we know that there exists $a<\frac{1}{2}$ such that 
\begin{equation*}
\mu_1(g_a)>M,
\end{equation*}
with $||g_a||_{L^{\infty}}=1$ and $||g_a||_{L^1}>1-2a$. From Lemma \ref{lLBMU} we know that for $\epsilon$ small enough we have that 
\begin{equation*}
\mu_1(\Omega_{\epsilon, g_a})\geq \mu_1(g_a)+O(1)\geq M.
\end{equation*}
This concludes the proof because $M$ is arbitrary and, by construction, for all $\epsilon$ the sets that we constructed have diameter equal to $1$.
\bigskip\noindent

{\bf Acknowledgements}: 
The authors are grateful to P. Freitas for mentioning us the reference \cite{K99} and to I. Polterovich for pointing out  
the reference \cite{HKP16}. This work was partially supported by the project ANR-18-CE40-0013 SHAPO financed by the French Agence Nationale de la Recherche (ANR).

\bibliographystyle{abbrv}
\bibliography{Refer1}

\end{document}